\documentclass[preprint,12pt]{elsarticle}

\usepackage{amssymb}
\usepackage{amsthm}

\usepackage{amsmath,amssymb,times,bbm,epsfig,float,fancyhdr}
\usepackage{verbatim}
\usepackage{algorithmic}
\usepackage{algorithm}

\usepackage{color}

\newcommand{\fracd}[2]{\displaystyle{\frac{{\displaystyle{#1}}}{{\displaystyle{#2}}}}}

\newcommand{\bd}{\boldsymbol{u}}
\newcommand{\JJ}{\boldsymbol{J}}
\newcommand{\J}{J}

\newcommand{\RP}{\mathrm{RP}}

\newcommand{\phic}{\phi^{\mathrm{\star}}} % subindex ?
\newcommand{\Phic}{\Phi^{\mathrm{\star}}} % subindex ?
\newcommand{\phis}{\phi^{\mathrm{x}}} % {\phi_{\backslash}}
\newcommand{\Phis}{\Phi^{\mathrm{x}}} 
\newcommand{\Phimax}{\Phi^{\infty}} %{\Phi^{\max}}
\newcommand{\phimax}{\phi^{\infty}} % {\phi^{\max}}
\newcommand{\Phisigma}{\Phi^{\sigma}} 
\newcommand{\phisigma}{\phi^{\sigma}} 
\newcommand{\Phiinf}{\Phi^{\diamond}} 
\newcommand{\phiinf}{\phi^{\diamond}} 

\newcommand{\phiu}{\phi^{u}} 
\newcommand{\uu}{\mbox{{\boldmath $u$}}}

\newtheorem{definition}{Definition}
\newtheorem{lemma}{Lemma}
\newtheorem{theorem}{Theorem}

\newtheorem{corollary}{Corollary}

\journal{Journal of Differential Equations}

\begin{document}

\begin{frontmatter}



\title{Contact Manifolds in a Hyperbolic System \\ of Two Nonlinear Conservation Laws}


\author[uct]{Stefan Berres}
\address[uct]{Departamento de Ciencias Matem\'{a}ticas y F\'{\i}sicas, Facultad de Ingenier\'{i}a, \\
         Universidad Cat\'olica de Temuco,
         Temuco, Chile.}
%  \email{sberres@uct.cl } }

\author[itam]{Pablo Casta{\~n}eda}
\address[itam]{Instituto Tecnol\'ogico Aut\'onomo de M\'exico \\
         R\a'io Hondo No. 1, Col. Progreso Tizap\'an,
         M\'exico D.F. 01080, M\'exico. } % \\
% \email{pablo.castaneda@itam.mx} }

\begin{abstract}

This paper deals with a hyperbolic system of two nonlinear conservation laws, where the phase space contains two contact manifolds.
The governing equations are modelling bidisperse suspensions, which consist of two types of small particles 
that are dispersed in a viscous fluid
and differ in size and viscosity.
% During sedimentation, the different particle species segregate creating areas of different composition. Spatially one-dimensional mathematical models of this process can be expressed as strongly coupled, nonlinear systems of first-order conservation laws.
% The solution of these equations is represented by the vector of volume fractions of each species as a function
% of depth and time. %, which will in general be discontinuous. 
%
% The system might fail to be strictly hyperbolic in models for polydisperse suspensions with more than two phases; 
% these can be modeled by the Masliyah-Lockett-Bassoon (MLB) flux vector, where the particles have the same density, 
% and the hindered-settling factor (a multiplicative algebraic expression appearing in the flux vector) depends on the
% particle size.
%
% Even though in the interior of the phase space strict hyperbolicity can be guaranteed,
% for the considered model equations for bidisperse suspensions with  size-dependent hindered settling factor,
For certain parameter choices quasi-umbilic points and a contact manifold in the interior of the phase space are detected.
%
% A small perturbation of initial data around this contact manifold leads to a drastic change of the solution structure.
%
% It is the purpose 
% In this contribution the impact of the contact manifold on the solution structure is elaborated.
% of Riemann problems that occur in standard batch settling tests.
%
% This examination is carried out along $2 \times 2$ systems  modelling bidisperse suspensions by the MLB model with particle-size dependent hindered settling factors.
%
% A characterization of the contact manifold reveals that it forms part of the Hugoniot locus of the origin
% which is moreover in a double-contact configuration, % DOUBLE-CONTACT
% and can be associated to the second characteristic family. 
%
%
%
%
% {\color{red}
% Moreover, the contact manifold is also in a double-contact configuration with the origin of the phase space.
% }
%
%{\color{blue} Que significa "with the state at the origin in the phase space"?}
% The characterization of solutions strongly depends on the properties of this special manifold.
The dependance of the solutions structure on this contact manifold is examined.
% }
The elementary waves that start in the origin of the phase space are classified. % turn out to be shocks belonging to different classes.
Prototypic Riemann problems that connect the origin to any point in the state space 
and that connect any state in the state space to the maximum line 
are solved semi-analytically. % The Riemann problems connecting are solved numerically.

\end{abstract}


\end{frontmatter}

% \linenumbers

% \section{}
% \label{}

\section{Introduction} \label{sec:intro}

% \subsection{Scope}

% \marginpar{... something about  modelling background}
Polydisperse suspensions can be described by balance equations as $N$ superimposed continuous phases, where particles of species $i$,
associated with a volume fraction $\phi_i$, distinguish in properties like size, density and viscosity~\cite{bbkt,bktw}; 
models with similar solution structure describe traffic and pedestrian flows \cite{bgc03,nhm11}.
For the considered model of bidisperse suspension, the solution structure of the solution to the initial value problem
 for standard batch settling tests
has been studied for the cases when strict hyperbolicity is assured~\cite{bbrp} and when the phase space provides elliptic regions~\cite{bbacta}. % In this context, 
The focus of this contribution is on the impact of a contact manifold in the interior of the phase space, that emerges for certain parameter settings.
This contact manifold has the physical property of coinciding particle settling velocities, and thus has a practical relevance, since one general goal in the process control of solid-liquid separation processes is to reduce segregation effects~\cite{biesheuvel00}. % A specific application is ceramic casting, where the goal is to obtain specific material mixtures that maintain a desired equilibrium~\cite{ozgur10}.

The generic form of kinematic sedimentation models for polydisperse
suspensions consists of the system of $N$ first-order hyperbolic equations
\begin{align} \label{systeq}
    \partial_t \phi_i + \partial_x f_i(\Phi) = 0, \quad f_i(\Phi) = \phi_i
    v_i (\Phi) , \quad i=1, \dots, N ,
\end{align}
where $t$ is time, $x$ is depth and the velocity components $v_i (\Phi)$ depend on the
concentration vector $\Phi = (\phi_1,\dots,\phi_N)^{\mathrm{T}}$.
% , which for $N=1$ reduces to the scalar equation \eqref{scalar}.
%
The unknown $\Phi$
denotes the vector of volume fractions of the solids phases
and is contained by the phase space of physically relevant concentrations
\begin{eqnarray}
\mathcal{D}_{\Phimax} := \bigg\{
 \Phi = (\phi_1, \dots, \phi_N)^{\mathrm{T}} \in
 \mathbb{R}^N:
 \begin{array}{l}
\phi_1 \geq 0, \ldots, \, \phi_N \geq 0,  \\
 \phi := \phi_1 + \ldots + \phi_N \leq \phimax 
 \end{array}
 \bigg\},
%\begin{split} %  \label{dphimaxdef}
% \mathcal{D}_{\Phimax} := \bigl\{
% \Phi = (\phi_1, \dots, \phi_N)^{\mathrm{T}} \in
% \mathbb{R}^N: \\
% \ \phi_1 \geq 0, \ldots, \phi_N \geq 0,  \
% \phi := \phi_1 + \ldots + \phi_N \leq \phimax \bigr\},
%\end{split}
\label{phasespace}
\end{eqnarray} 
where the total concentration $\phi := \phi_1 + \ldots + \phi_N$ is bounded from above by the maximum packing
concentration $\phimax$. The maximum packing manifold 
\begin{align}
\label{maxphi}
    \partial^{\infty} := \{ \Phi = (\phi_1, \dots, \phi_N)^{\mathrm{T}} : \phi := \phi_1+ \dots + \phi_N=\phimax \} .
\end{align}
% $\partial^\infty$ 
is the set of all maximal states.

% \subsection{Model with particle-size specific hindering factors} \label{sec:model}

% \end{definition}
% We denote with $\mathcal{D}_{\Phimax}^0$ to the interior domain.

% The interior of $\Phi \in \mathcal{D}_{\Phimax}$ is denoted by
% $\mathcal{D}_{\Phimax}^0$.
 
%
The resulting system of conservation laws is actually
a system of mass balances for different solids species, 
where the nonlinear flux function 
\begin{align*}
	\boldsymbol{f}(\Phi) := (f_1(\Phi), \dots, f_N(\Phi))^{\mathrm{T}}
\end{align*}
can be derived from the corresponding momentum balances \cite{bbkt, bktw, masl}
The components describe the flow process of the dispersed solids
phases in a liquid, where the dispersed phases are considered as a continuum.
%
% The nonlinear flux function has the % Cartesian 
The flux function has components
\begin{align} \label{fluxfunction}
  f_i(\Phi) = \phi_i v_i(\Phi), \qquad
  v_i(\Phi) = u_i(\Phi) - \Phi^{\mathrm{T}} \uu, \;
  %  \uu = (u_1(\Phi), \dots, u_N(\Phi))^{\mathrm{T}} ,  
\quad i=1, \dots, N,
\end{align}
with $\uu = (u_1(\Phi), \dots, u_N(\Phi))^{\mathrm{T}}$, 
where the absolute velocity
$v_i =v_i(\Phi)$
of a representative solids particle
depends on a linear combination of
 the solid-fluid relative (``slip'') velocities
\begin{align} \label{relvel}
    u_i(\Phi) := v_i(\Phi) - v_{\mathrm{f}},
\end{align}
which are relative to the fluid velocity
$v_{\mathrm{f}}$.
The flux function model
\eqref{fluxfunction} is closed by specifying the relative velocity as
\begin{align} \label{relvelcon}
    u_i(\Phi) = v_{\infty i} V_i(\Phi) , 
\end{align}
where the constant $v_{\infty i}$ is the Stokes velocity, which
quantifies the settling velocity of a single particle in a fluid,
and $V_i(\Phi)$ is the hindered-settling velocity % that decreases with increasing concentrations,
that is an non-increasing function of the components of $\Phi$, see~\cite{bbb}.
Following Richardson and Zaki \cite{rz},
the hindered-settling velocity is set as
\begin{align}
\label{videf} V_i (\Phi) % = V_i (\phi) 
    := \begin{cases}
(1-\phi)^{n_i-1} & \text{if $0 \leq \phi \leq \phimax$,} \\
 0 & \text{otherwise,}
 \end{cases}
        \quad i = 1, \dots, N , % , \quad
\end{align}
% where $\phimax$ stands for the maximum packing. 
% The terminal velocity~$v_{\infty i}$ describes the velocity at very small concentrations and 
where the exponent $n_i > 1$ accounts for the slow down of the process at increasing concentrations.
% can be determined from measurements. 
% of boundary settling velocities at different initial
% concentrations, as obtained in the batch sedimentation of a single species.
Assumptions \eqref{relvelcon} and \eqref{videf}
can be combined as 
\begin{align} \label{uidef}
	u_i(\Phi) = v_{\infty i}(1-\phi)^{n_i-1}
\end{align}
for $\Phi \in \mathcal{D}_{\Phimax}$.

Strictly hyperbolic systems of conservation laws, where the eigenvalues of the Jacobian matrix of the flux function are real and distinct,
provide a relatively well understood framework for the solution of Riemann problems~\cite{dafermos}.
In \cite{bktw}, strict hyperbolicity of the system \eqref{systeq} with flux function \eqref{fluxfunction}
has been first shown for $N=2$ and later on in~\cite{bbkt} for general $N$,
but up to then only for coinciding hindered-settling factors $V_1(\phi)=V_2(\phi)=\dots=V_N(\phi)$,
which depend on the total concentration $\phi$. % , but do not take the more general form \eqref{videf}. % , which allows a dependence on the whole vector $\Phi$.
For a model with the more general hindrance factor \eqref{videf}, this implies to have constant exponents $n_1=n_2=\dots=n_N$,
a restriction that turned out to be unnecessary:
In \cite{bbb}, it is shown that strict hyperbolicity also holds for general $N\ge2$
and relative velocities of form $V_i(\Phi) = V_i(\phi)$ % \eqref{videf} % \eqref{relvelcon},
as long as the inequality 
\begin{align} \label{uineq}
	u_i' (1 - \phi) - u_i < 0
\end{align}
holds for all $i=1,\dots,N$, where $u_i'$ denotes the derivative with respect to $\phi$; 
in this situation the only restriction on the hindered-settling function $V_i(\Phi)=V_i(\phi)$
is to depend on the total concentration $\phi$.
The inequality~\eqref{uineq} 
is satisfied in particular when the relative velocities are ordered as $u_1 > u_2 > \dots > u_n$ for any $\phi$.
This holds in the case of the hindered-settling function \eqref{videf} for the situation
if the parameters are ordered as
\begin{align}  \label{vorder}
	v_{\infty 1} >  v_{\infty 2} > \dots > v_{\infty N} \quad \text{with} \quad n_1 < n_2 < \dots < n_N.
\end{align}
In \cite{rosa}, a secular equation framework was established that allows to verify strict hyperbolicity by checking a simple algebraic criterion;
this framework has been applied to the considered model with a Richardson-Zaki hindered settling function having constant exponent.
In \cite{bv09}, this framework was adapted to the same general model setting as in \cite{bbb}, {\it i.e.} with size-dependent hindered settling factors 
that not necessarily take the form \eqref{videf}.
Subsequently, in \cite{bdmv} the secular equation framework was applied to a series of choices of hindered-settling functions. % COPY

Preliminary numerical simulations of Riemann problems for $N=3$ with arbitrary parameter choices (not exposed here)
showed that strict hyperbolicity might fail as coincidence of eigenvalues occurs,
providing an abrupt change of the solution structure.
The fact that this phenomenon of abrupt change already appears for $N=2$
made us to look for analytical insights in this situation, which lead to the present contribution.

This contribution deals with the wave classification 
for $2\times 2$ systems of conservation laws
that arise as one-dimensional kinematic models for the sedimentation of bidisperse suspensions.
The analytical examination of bidisperse suspensions 
gives insights to flow properties of polydisperse suspensions, which 
are mixtures of small solid particles dispersed in a viscous fluid.
In this contribution the focus is on models for particle suspension where all particles are assumed to have the same density.
% The sizes are represented by a diameter $d_1> d_2 > \dots > d_N$ and the viscosity by exponents $n_i$. {\bf introduce exponents before...}
%
% If $v_i$ is the phase velocity of species $i$ then the continuity equations of the $N$ species are $\partial_t \phi_i + \partial_x (\phi_i v_i)=0$ for $i=1,\dots,N$
% 
Specifically, in this contribution, % Section 
the properties of the $2 \times 2$ system \eqref{systeq}
(with $N=2$), flux function \eqref{fluxfunction} and closures \eqref{relvelcon}, \eqref{videf}
are studied.
The model of our interest contemplates the following specifications,
which is done in opposition to \eqref{vorder}, which would guarantee strict hyperbolicity: % restrictions:

(S1) $v_{\infty 1} \,\,\,>\,\,\, v_{\infty 2} \,\,\, > \,\,\, 0$,% \,\,\,>\,\,\, v_{\infty 3}$,

(S2) $n_1 \,\,\,>\,\,\, n_2 \,\,\,>\,\,\, 1$, %n_3 \,\,\,>\,\,\, 1$,

(S3) $\phimax \,\,\equiv\,\, 1$.

% \noindent
% Note that 

% \begin{align} \label{parameterinequ}
%	n_1 > n_2
%	\quad
%	 and 
%	 \quad v_{1\infty} > v_{2 \infty}
% \end{align}

For $N=2$ and $V_i(\phi)$ given by \eqref{videf} the flux function $\boldsymbol{f}(\Phi)=(f_1(\Phi), f_2(\Phi))^{\mathrm{T}}$
takes the form
\begin{align*} % \begin{gather} % \label{n2flux}
	f_1(\Phi)=\phi_1 \Bigl( v_{\infty 1}(1-\phi_1) (1-\phi)^{n_1 - 1}  - v_{\infty 2} \phi_2 (1-\phi)^{n_2-1} \Bigr),�\\
	f_2(\Phi)=\phi_2 \Bigl( v_{\infty 2}(1-\phi_2) (1-\phi)^{n_2 - 1}  - v_{\infty 1} \phi_1 (1-\phi)^{n_1-1} \Bigr),
% \end{gather}
\end{align*}
for values $\Phi \in \mathcal{D}_{\Phimax}$ and $f_1(\Phi)=f_2(\Phi)=0$ otherwise.
%
% it formulated as an own lemma.
% \begin{lemma} 
%
% This article deals with 
%
A special interest consists 
in the classification of the solution
structure of the Riemann problem
\begin{align} \label{rpdef}
    \Phi(t=0,x) = \begin{cases}
        \Phi^- \quad \text{if} \quad x < 0, \\
        \Phi^+ \quad \text{if} \quad x > 0 .
    \end{cases}
\end{align}
% for the considered $2 \times 2$ system. % and $3 \times 3$ systems. 
For convenience,
the Riemann problem consisting of the system of PDEs \eqref{systeq} with initial condition \eqref{rpdef} is referred to as $\mathrm{RP}(\Phi^-,
\Phi^+)$, with left and right values $\Phi^-$ and $\Phi^+$, respectively, to be specified.
%

% \subsection{Outline of the paper}

The application of this model is the batch settling process of an initially homogeneous suspension in a closed container
described by
the initial-boundary value problem % system of conservation laws
 \begin{gather} \nonumber % \label{firstordersystem}
  \partial_t \phi_i + \partial_x f_i (\Phi)
  = 0, \quad i=1, 2, \\
    \Phi(0,x)= \Phi_0(x), \quad  0 \leq x \leq
    L,
\label{pc-initcond} \\
    f_i (\Phi)
    =0, \quad
    x \in \{0,L\}, \quad i=1, 2 ,
    \label{boundcond}
\end{gather}
where $L$ is the domain height and the components of the flux-density vector $\boldsymbol{f}(\Phi) = ( \smash{f_1 } (\Phi), \smash{f_2 } (\Phi))^{\mathrm{T}}$
are given by \eqref{fluxfunction}.
Because of the zero-flux boundary
condition \eqref{boundcond}, the initial-boundary data \eqref{pc-initcond} and
\eqref{boundcond} can be replaced by the Cauchy data
\begin{align}
  \Phi (0,x) = \Phi_0 (x) = \begin{cases}
 O & \text{for $x <0$,} \\
\Phi_0 & \text{for $0 \leq x \leq L$,} \\
\Phi^{\infty}    & \text{for $x > L$,}
\end{cases}
 \label{inithom}
\end{align}
where $O := (0,\,0)^{\mathrm{T}}$ is the origin and $\Phi^{\infty}$ is a state on the maximum concentration manifold \eqref{maxphi}.
%\begin{align} \label{maxphi}
%    \partial^{\infty}:= \{ \Phi=(\phi_1, \phi_2) \in %\mathbb{R}^+_0 \times \mathbb{R}^+_0:
%    \phi = \phi_1+\phi_2= \phimax \}.
%\end{align}
%
Therefore, the Riemann problems $\mathrm{RP}(O,\Phi)$ and $\mathrm{RP}(\Phi,\Phimax)$
 are of particular interest. % since they form the basis of the initial value problem \eqref{inithom}
This contribution reveals analytical insights into the solution structure of the Riemann problem $\mathrm{RP}(O,\Phi)$.
From an application point of view, the Riemann problem $\mathrm{RP}(O,\Phi)$ describes 
the interactions on the upper interface between clear liquid and initially homogeneous suspension 
during a batch settling process.

% \marginpar{Add something more general about Riemann, Lax, Oleinik?}
%
% \marginpar{rearrange introduction}

% \marginpar{mejorar titulo de la seccion}

% \ref{sec:intro}
% \ref{sec:model} 

% \marginpar{check consistency / skip?}
%In Section \ref{sec:contact} the contact manifolds are characterized,
%
%in Section \ref{sec:hugloc} the Hugoniot of the origin is determined.
%
% In Section \ref{sec:eig}
% the eigenvalues of the Jacobian matrix of the flux function are calculated, allowing to complete the characterization of the contact manifolds. Eigenvectors are calculated and integral curves visualized.
%
% In Section \ref{sec:shock} the shocks types for states on the Hugoniot locus of the origin are classified.
%
% In Section \ref{sec:rp} the structure of solutions of Riemann problems are determined.

% \marginpar{Lugar optimo? Introducir al parrafo...}

% \subsection{Related work}
%

With respect to related work,
several studies on weakly hyperbolic systems, {\it i.e.} systems that are hyperbolic but not strictly hyperbolic,
are developed for models of multi-phase flow in porous media. 
%
% \marginpar{En el sistema no tenemos perdida de hiperbolicidad, entonces este parrafo no es pertinente. TENEMOS en los quasi-umbilicos}
%
For three-phase flow in porous media, the Corey model with convex permeability leads to a single isolated point, 
the so-called umbilic point, in which strict hyperbolicity fails \cite{HYPcfm, impt88, impt92, Matos, ss87}; 
another loss of hyperbolicity occurs when a the phase space contains an elliptic region as it occurs for the Stone model~\cite{fayers}. To solve Riemann problems, the wave-curve method has been applied, in which a sequence of elementary waves are connected. When strict hyperbolicity fails, it is not sufficient to consider the method of Liu to deal with non-convex fluxes; rather one has also to considers non-local branches of the Hugoniot locus \cite{cfm}. The wave-curve method has been applied to the injection problem, where a gas-water mixture is injected in a porous medium containing oil \cite{asfmp10}.
For a system of two conservation laws with a quadratic flux function the solution in the neighborhood of the umbilic point has been classified in \cite{Cido, Matos, ss87}.
Following the idea of studying the solution behavior in the neighborhood of an umbilic point, in the case of a quadratic flux function four different types of umbilic points corresponding to different shapes of the close by integral curves could be identified \cite{ss87}.
In the Corey model with convex permeability only two types of umbilic points occur~\cite{Matos}.
According to the proposed classification, certain types of Riemann problems have been considered in \cite{impt88}.
A systematic classification of solutions of the Riemann problem for non-strictly hyperbolic systems of two conservations laws, which count with an umbilic point and with the identity viscosity matrix has been carried out in \cite{smp96,spm01}.
A non-local Hugoniot locus leads to non-classical waves and in some cases to transitional shocks \cite{ampz02capillary,majda}. These transitional shocks are sensitive to the regularization by a non-identical viscosity matrix.

% This article is a continuation of a series of articles on the phase
% space structure of a model for polydisperse suspensions. \cite{bbb} (...)

\section{Basic definitions} % General theory for Riemann problems - Definition of local properties of weakly hyperbolic systems}
\label{sec:theory}

In this paragraph several definitions~\cite{Cido, furtado, panters} 
% [of local properties of weakly hyperbolic systems]
are collocated in order to facilitate the appropriate classification for the system under study.

%

% \subsection{(Other definitions)}

\begin{definition} % [Hugoniot locus]
The Hugoniot locus of a state $\Phi^-$, denoted as $\mathcal{H}(\Phi^-)$, is the set of all states $\Phi^+$ 
that satisfy the Rankine-Hugoniot condition
\begin{align} \label{rhcond}
	f(\Phi^+) - f(\Phi^-) = \sigma ( \Phi^+ - \Phi^-),
\end{align}
where $\sigma = \sigma(\Phi^-,\,\Phi^+)$ is the propagation velocity of the discontinuity.
% and the flux functions $\boldsymbol{f}(\Phi)$ and concentrations $\Phi$ are given as before. 
% \marginpar{$s\mapsto\sigma$, $s$ es velocidad $v_i$}
%for some $s \in \mathbb{R}$.
\end{definition}
% Notice that $\Phi^+$ belongs to $\mathcal{H}(\Phi^-)$ if and only if $\Phi^-$ belongs to $\mathcal{H}(\Phi^+)$.

%\TD{How are subsets of a Hugoniot locus referred to? Are they really used? Creo que SI, es dado por LAX, mira mas adelante}
The shock classification according to Lax~\cite{Lax57} is used
in order to refer to a subset of the Hugoniot locus that corresponds to a certain wave family.
% \color{red} 
% In order to do that, we will use the 
%
% 
% \begin{definition}
% The subset of the Hugoniot locus $\mathcal{H}(O)$ whose states are connected to the origin by 
% 1-Lax shocks is called {\bf 1-Lax locus} of the origin and denoted by $\mathcal{L}_1(O)$.
% \end{definition}
% \marginpar{La notacion $\mathcal{L}_1(O)$ no es usada}
% 
%
% Solutions of conservation laws generally provide jump discontinuities.
%The Hugoniot locus of a point $\Phi^-$, denoted as $\mathcal{H}(\Phi^-)$, is given by all the points $\Phi^+$ that satisfy the Rankine-Hugoniot condition \eqref{rhcond}
The corresponding admissible shocks are classified depending on inequalities between the first and the second eigenvalues at both sides of the discontinuity and the discontinuity speed itself, see {\it e.g.} \cite{Lax57, smp96}.

\begin{definition}\label{def:shocks}
Three kinds of classical admissible shocks can be distinguished.
The classification applies between a left state $\Phi^-$ and a right state $\Phi^+$ 
which are connected by the Rankine-Hugoniot condition \eqref{rhcond}
with jump velocity $\sigma = \sigma(\Phi^-,\,\Phi^+)$:

{\it 1-Lax shock}: % Inequalities 
$\lambda_1(\Phi^+) \leq \sigma \leq \lambda_1(\Phi^-)$ and $\sigma \leq \lambda_2(\Phi^+)$ % hold.

{\it 2-Lax shock}: % Inequalities 
$\lambda_2(\Phi^+) \leq \sigma \leq \lambda_2(\Phi^-)$ and $\lambda_1(\Phi^-) \leq \sigma$ 
% hold.

{\it Over-compressive shock (OC)}: % Inequalities 
$\lambda_2(\Phi^+) < \sigma < \lambda_1(\Phi^-)$ % hold.

% No creo que necesitemos de esta definicion
% {\it Transitional shock}. Inequalities $\lambda_1(\Phi^-) < \sigma < \lambda_2(\Phi^-)$ and $\lambda_1(\Phi^+) < \sigma < \lambda_2(\Phi^+)$ hold.
\end{definition}

\noindent
Left- and right-characteristic shocks are included in this shock type definition.
They occur when a shock speed coincides with the characteristic speed.
 % for the simpler cases below.) 
Whereas by definition an over-compressible shock cannot be characteristic, the limit of the inequalities above are included in the {\it 1-Lax} and {\it 2-Lax} shocks, which are also called first and second shock waves.

An inflection manifold $\mathcal{I}_i$ is determined for all states $\Phi$ where the $i$-th eigenvalue attains a maximum or minimum value along the integral curve of the same family.
\begin{definition}[Inflection curve]\label{def:inflection}
The $i$-th inflection manifold is defined as
\begin{equation*}
\mathcal{I}_i := \big\{ \Phi\in\mathcal{D}_{\Phimax} \,:\, \nabla\lambda_i(\Phi)\cdot r_i(\Phi) = 0 \big\} ,
\end{equation*}
where $\lambda_i$ is the $i$-th eigenvalue 
of the Jacobian matrix of the flux function
and $r_i$ is the corresponding eigenvector.
\end{definition}

% Two points $\Phi^-,\,\Phi^+$ of the phase space belong to a double
% contact discontinuity if $(\Phi^-,\,\Phi^+)$ is a solution of the
% Rankine-Hugoniot condition \eqref{rhcond}, {\it i.e.}, $\Phi^-\in\mathcal{H}(\Phi^+)$.

% {\color{red}
In the sense of the invariant manifolds defined in \cite{Tem83}, we introduce the following concept

\begin{definition}\label{def:ContactMan}
An $i$-th contact manifold occurs when the $i$-th integral curve passing through a state $\Phi^o$ 
coincides with a part of the Hugoniot locus $\mathcal{H}(\Phi^o)$,
such that any state $\Phi$ on this intersection satisfies
$$\lambda_i(\Phi^o) = \sigma(\Phi^o,\,\Phi) = \lambda_i(\Phi),$$
for the shock speed $\sigma(\Phi^o,\,\Phi)$.
\end{definition}

A necessary condition for establishing an $i$-th contact manifold
is that the integral curve is not a rarefaction in the usual sense, but that the characteristic speed is fixed along the curve.
For states on a contact manifold the following transitivity rule holds.
If $\Phi_1$ and $\Phi_2$ mutually belong to the Hugoniot locus of the other,
connected by a shock of speed $\sigma = \sigma(\Phi_1,\,\Phi_2)$ and
if $\Phi_2$ is on the Hugoniot locus of a state $\Phi_3$ by a shock of the same speed $\sigma$,
then $\Phi_1$ belongs to $\mathcal{H}(\Phi_3)$ and $\sigma(\Phi_1,\,\Phi_3)$ also coincides with $\sigma$.
This is the essence of the Triple Shock Rule \cite{HYPcfm,cfm,furtado}. Another useful version establish the following.

\begin{lemma}\label{TripleShock}
Let $\Phi_1,\,\Phi_2,\,\Phi_3$ be non-collinear states such that $\Phi_1,\,\Phi_2$ belong to $\mathcal{H}(\Phi_3)$ and $\Phi_1$ belongs to $\mathcal{H}(\Phi_2)$, then $\sigma(\Phi_1,\,\Phi_2) = \sigma(\Phi_2,\,\Phi_3) =\sigma(\Phi_1,\,\Phi_3).$
\end{lemma}

% }

% % \marginpar{No confundir contacto y contacto-doble, siempre se hizo la diferencia?}
% %
% %\TD{What is the difference between double-contacts and contact manifolds?}
% %\TD{Is the transitivity rule formulated fine?}
% %
% For states on a contact manifold the following transitivity rule holds.
% If $\Phi_1,\,\Phi_2$ are located on the same contact manifold, connected by 
% a shock of speed $\sigma = \sigma(\Phi_1,\,\Phi_2)$ and $\Phi_2$ is connected to
% a state $\Phi_3$ by a shock of the same speed
% % on the same manifold, then the shock speed is the same, 
% then  $\Phi_1$ and $\Phi_3$ also belong to the same contact configuration.
% %
% % A smooth set of double-contacts forms a manifold with constant characteristic speed.

On a contact manifold, rarefactions are indistinguishable from shocks; all speeds match a characteristic speed. Remarkably, a Hugoniot locus with constant characteristic speed is planar and coincides with the integral curves \cite{Tem83}.

\section{Contact manifold} \label{sec:contact}

If the specifications (S1) and (S2)
of the considered model hold,
then a contact manifold inside the phase space can be identified.
%
% % the interior 
% {\color{blue} 
This manifold turns out to be decisive for the characterization of solutions of Riemann problems,
in particular because the origin is connected to this manifold by a
% }
% {\color{red}
right characteristic shock.

\begin{definition}
A set
$\mathcal{C}(\Phic)$ is defined as % a connected manifold 
a subset of the phase space $D_{\Phimax}$ % $\mathcal{C}_{\Phi^{\max}}$
which contains a state $\Phic=(\phic_1,\phic_2)$ 
% and has the property
that is connected to other states by the property
\begin{align} \label{csdef}
	\mathcal{C}(\Phic) := \{ \Phi \in D_{\Phimax}: v_1(\Phi) = v_2(\Phi) = v_1(\Phic) \} .
\end{align}
% where $s=v_1(\Phic)$.
\end{definition}
The state $\Phic$ is a representative of the set $\mathcal{C}(\Phic)$.
The definition of a set $\mathcal{C}(\Phic)$ unifies two complementary properties:
\begin{enumerate}
\item $v_1(\Phi)=v_2(\Phi)$ for all $\Phi \in \mathcal{C}(\Phic)$,
\item $v_i(\Phi^-)=v_i(\Phi^+)$ for $\Phi^-, \Phi^+ \in \mathcal{C}(\Phic)$ and $i =1,\,2$.
\end{enumerate}
Property (1) describes the local coincidence of velocities of different phases in one particular state,
whereas property (2) describes the constancy of the velocities of a single family along the manifold.

The following Lemma is a generic result, stating that,
a set $\mathcal{C}(\Phic)$ is a contact manifold
whenever the flux function has a certain structure.
%  \eqref{fluxfunction}.

\begin{lemma} \label{lemcont}
If the flux function of the system \eqref{systeq} has the structure 
\begin{align} \label{fpv}
	f_i(\Phi)=\phi_i v_i(\Phi)
\end{align}
then, for any $\Phic \in D_{\Phimax}$, the set $\mathcal{C}(\Phic)$
is a contact manifold with constant shock speed 
\begin{align} \label{sv1p}
	\sigma(\Phi^-, \Phi^+) = v_1(\Phic) .
\end{align}
Moreover, if the structure of the flux function \eqref{fpv} is such that the
absolute velocity $v$ depends on the total concentration $\phi$, namely
\begin{align} \label{vphi}
	v_i(\Phi) = v_i(\phi),
\end{align}
then the contact manifold $\mathcal{C}(\Phic)$ 
consists of the line 
\begin{align} \label{cline}
	\mathcal{C}(\Phic) = \{ \Phi \in \mathcal{D}_{\Phimax}: \phi_1+\phi_2 = \phic \} ,
\end{align}
where $\phic=\phic_1 + \phic_2$ is the total concentration of $\Phic$.
% Thus, any set of states satisfying the Definition \eqref{csdef} of $\mathcal{C}(\Phic)$
% also satisfies the property \eqref{uvequiv}.
\end{lemma}

\begin{proof}
% Since $\Phic$ is contained in the set 
% $ We set the velocity as $s:=v_1(\Phic) = v_2(\Phic)$.
%
By definition \eqref{csdef}
any states $\Phi^-, \Phi^+ \in \mathcal{C}(\Phic)$  satisfy that
$$v_1(\Phi^-)=v_2(\Phi^-)=v_1(\Phi^+)=v_2(\Phi^+)=v_1(\Phic) . $$ % =s$
%
% Together with structure  \eqref{fpv}
such that one can factorize 
\begin{align} \label{pvs}
	\phi_i^+ v_i(\Phi^+) - \phi_i^- v_i(\Phi^-)
	% = \phi_i^+  v_1(\Phic) - \phi_i^-  v_1(\Phic)
	=  v_1(\Phic) (\phi_i^+ - \phi_i^-) , \quad i=1,2 .
\end{align}
From the specific structure of the flux function \eqref{fpv} % , where $f_i(\Phi)=\phi_i v_i(\Phi)$, % \eqref{fluxfunction}
one recognizes that equation \eqref{pvs} states the Rankine-Hugoniot condition \eqref{rhcond} 
with shock speed \eqref{sv1p}. % $\sigma=\sigma(\Phi^-,\,\Phi^+) = v_1(\Phic)$.
This establishes that 
all states on a set $\mathcal{C}(\Phic)$ belong mutually to the Hugoniot locus of each other
and
any two states on a set $\mathcal{C}(\Phic)$ can be connected by a shock of the same speed. % $\sigma$

%
%

% \marginpar{check the argument \\ Direct from Eq.\eqref{videf} into the Jacobian}
% \underline{
Since shock speeds locally converge to an eigenvalue, % [Dafermos 2000, Theorem 8.2.1; p. 149]
on a manifold with constant shock speed 
any two states $\Phi^-$ and $\Phi^+$ have an eigenvalue that coincides with the shock speed,
\begin{align*}
	\lambda_i(\Phi^-) = \lambda_i(\Phi^+) = v_1(\Phic) ,
	% \rightarrow \sigma(\Phi^-, \Phi^+) = v_1(\Phi^-), \quad \text{as} \quad \Phi^+ \rightarrow \Phi^-
\end{align*}
% that assumes the common speed $v_1(\Phi^-)$
such that Definition~\ref{def:ContactMan} of a contact manifold
%\begin{align}
%	\lambda_i(\Phi^-) = \sigma(\Phi^-, \Phi^+) = \lambda_j(\Phi^+) % = v_1(\Phic)
%	% \rightarrow \sigma(\Phi^-, \Phi^+) = v_1(\Phi^-), \quad \text{as} \quad \Phi^+ \rightarrow \Phi^-
%\end{align}
follows. % {\color{red} (The lemma is proofed rigorously after Corollary~\ref{cor:contact}.)}
%
%\TD{Clarify definition of contact manifold (versus double contact).}
%

The shape of a manifold $\mathcal{C}(\Phic)$ to be a line can be deduced from
property \eqref{vphi} which assures that the absolute velocities are
constant on the line \eqref{cline}, {\it i.e.}
$v_1(\Phi)=v_2(\Phi)$ for all $\Phi \in \mathcal{C}(\Phic)$.
\end{proof}

%
% $u_1(\Phi)=u_2(\Phi)$ for all $\Phi \in \mathcal{C}(\Phic)$,
% satisfying the Definition \eqref{csdef} of $\mathcal{C}(\Phic)$
% The characteristic speed $\lambda_2(\Phi) \equiv s$ follows from Theorem~\ref{thm:eigen}. \marginpar{Este thm antes!} % extension despues

With this Lemma, nothing yet is said about the existence of such a contact manifold for the considered model equation.
% In fact, the considered model equation provides two contact manifold.
Neither it is decided to which characteristic family the contact manifold belongs, whether to the first or the second one.
%\marginpar{decidir "first, second" vs. "small, large" vs "slow, fast" families}

A further characterization of a set $\mathcal{C}(\Phic)$
is developed in the sequel, starting from properties 
that can be derived from the generic structure of the model, leading to properties that depend on particular model specifications.
A property that can be used in several instances is
\begin{align} \label{uvequiv}
 	v_i(\Phi)=v_j(\Phi) \quad \Leftrightarrow	\quad u_i(\Phi)=u_j(\Phi),
	\qquad
	i,j \in \{ 1,\,2\} ,
\end{align}
which follows directly from \eqref{relvel}. % \marginpar{no existe referencia}
Property \eqref{uvequiv}
assures that condition \eqref{vphi} is satisfied for the model under consideration, 
% the system \eqref{systeq} with the given model specifications, namely  
where the flux function has the structure \eqref{fluxfunction} complemented by constitutive assumptions 
\eqref{relvelcon} and \eqref{videf}.

% the equality $v_i(\Phi)=v_j(\Phi)$ is equivalent to $u_i(\Phi)=u_j(\Phi)$.

% the determination of the absolute velocities  \eqref{absvel} in terms of the relative velocities, which leads to property \eqref{uvequiv},

% {\color{red}
%
%\marginpar{??? El Lemma dice nada sobre los $\lambda$ que aparecen en el Teorma. Ademas, no es obvio como se deduce que tenemos una "invariant manifold".}
%
% \begin{remark}
% Another proof of Lemma~\ref{lemcont} follows from Lemma~\ref{lem:contact}. As $\mathcal{C}(\Phic)$ is an invariant manifold
% (in the sense that solutions that start on the invariant manifold remain there for all time), Theorem 2 in \cite{Tem83} guarantees that $\mathcal{C}(\Phic)$ is a contact manifold.
% \end{remark}
% }

Up to now it has been shown that any set $\mathcal{C}(\Phic)$ is a contact manifold, which has in addition the shape of a line.
% but it is pending to establish that there exists such a set.
In the next Lemma it is shown that for the considered model specifications such manifolds effectively exist.

\begin{lemma}
If conditions (S1) and (S2) % \eqref{parameterinequ} 
are satisfied,
then two distinct contact manifolds exist in the domain $\mathcal{D}_{\Phimax}$, namely
\begin{eqnarray}
\label{def:contactMax}
\mathcal{C}(\Phimax) &:=& \partial^{\infty}= \{ \Phi = (\phi_1,\,\phi_2)^{\mathrm{T}}: \; \phi_1+\phi_2 = \phimax \}, \\
\label{def:contact*}
\mathcal{C}(\Phis)     &:=& \{ \Phi = (\phi_1,\,\phi_2)^{\mathrm{T}}: \; \phi_1+\phi_2 = \phis \} . 
\end{eqnarray}
The manifold 
% namely 
$\mathcal{C}(\Phimax)$
is represented by any state $\Phimax \in \partial^\infty$.
% and , where $\Phic$ is a state 
A representative state $\Phis$ 
% in the interior of $\mathcal{D}_{\Phimax}$
that defines the contact manifold $\mathcal{C}(\Phis)$
that is distinct to the manifold $\mathcal{C}(\Phimax)$ can be
identified as
\begin{align}\label{def:phi*}
	\Phis =(\phis,\, 0)^{\mathrm{T}} ,
	\quad
	\phis = 1 - (v_{2 \infty} / v_{1 \infty})^{1/(n_1 - n_2)} .
% 1 - \exp \Biggl( \fracd{\ln (v_{1 \infty} / v_{2 \infty})}{n_2 - n_1} \Biggr)
\end{align}

\end{lemma}

\begin{proof}
First, it is shown that the boundary $\partial^{\infty}$ is a contact manifold.
For any state $\Phimax \in \partial^{\infty}$
one has $u_1(\Phimax)=u_2(\Phimax)=0$. By property \eqref{uvequiv} % \eqref{absvel}
one gets $v_1(\Phimax)=v_2(\Phimax)=0$ for any state $\Phimax \in \partial^{\infty}$,
satisfying the Definition \eqref{csdef} of the set $\mathcal{C}(\Phic)$,
which by Lemma \ref{lemcont} is a contact manifold.

To show that there is an additional contact manifold $\mathcal{C}(\Phis)$,
% If \eqref{parameterinequ}
% If $n_1 > n_2$ and $v_{1\infty} > v_{2 \infty}$ 
one has to find a set of states $\Phis \not\in \mathcal{C}(\Phimax)$
such that $v_1(\Phis)=v_2(\Phis)$.
Because of property \eqref{uvequiv} % \eqref{absvel}
it is equivalent to find 
a $\Phis$ such that
$u_{1}(\Phis)=u_{2}(\Phis)$.
% Denote by $\phis$
Resolving 
\begin{align*}
	v_{1 \infty} (1-\phis)^{n_1} = v_{2 \infty} (1-\phis)^{n_2}
\end{align*}
with respect to $\phis$, which is the total concentration of any representant $\Phis$,
gives \eqref{def:phi*}.
The conditions on the parameters (S1) and (S2) % \eqref{parameterinequ}
guarantees that $\phis \in (0,1)$ exists.
% (Notice that both $(0,\,\phis)^{\mathrm{T}}$ and $(\phis,\, 0)^{\mathrm{T}}$ belong to $\mathcal{C}(\Phis)$.)
% The set $\mathcal{C}(\Phis)$ can be represented by $\Phis =(\phis,\, 0)^{\mathrm{T}}$.
\end{proof}

Throughout this work,
the notation $\mathcal{C}(\Phic)$ is used to refer to a generic contact manifold,
whereas $\mathcal{C}(\Phis)$ and $\mathcal{C}(\Phimax)$ refer to specific contact manifolds
with assigned representative values $\Phis$ and $\Phimax$, respectively.
%
% TRIPLE SHOCK RULE
%\marginpar{triple shock rule: this comes suddenly; need to put into context AGREE}
% \marginpar{Aqui estaba la TRIPLE SHOCK RULE, lo mande para adelante}
%
%
% \subsection{Hugoniot locus of origin} \label{sec:hugloc} % in the $2\times 2$ system}
% \TD2{'Deben de existir ramos no-locales del Hugoniot locus de puntos
% del interior del tetrahedro.': Teoria general y desarrollo para el modelo.}
%
% In this Section, the Hugoniot locus of the origin is characterized.
%
It turns out that the contact manifolds $\mathcal{C}(\Phis)$ and $\mathcal{C}(\Phimax)$  form transversal branches of the Hugoniot locus of the origin.

\begin{lemma}[Hugoniot locus of origin] \label{hlorigin2}
% For $N=2$ 
The Hugoniot locus $\mathcal{H}(O)$ of the origin $O = (0,\,0)^T$
% $\Phi^+ = O$ % O^{\mathrm{T}}$ 
consists of four branches: the two coordinate axes as local branches,
\begin{eqnarray}\label{def:boundaries}
\partial^1 := \{ \Phi = (\phi,\,0)^{\mathrm{T}}, \; \phi \in [0,1] \}, \quad 
\partial^2 := \{ \Phi = (0,\,\phi)^{\mathrm{T}}, \; \phi \in [0,1] \}, 	
\end{eqnarray}
with variable shock speed
\begin{align} \label{speedL}
	\sigma(O, \Phi) = v_i(\Phi) \quad \text{for} \quad \Phi \in \partial^i, \quad i=1,2 ,
\end{align}
and the two contact manifolds $\mathcal{C}(\Phis)$ and $\mathcal{C}(\Phimax)$, 
identified as \eqref{def:contactMax} and \eqref{def:contact*},
as transversal branches
with constant shock speed
\begin{align} \label{speedT}
	\sigma(O, \Phi) = v_1(\Phi) = v_2(\Phi) \quad \text{for} \quad \Phi \in \mathcal{C}(\Phic) .
\end{align}
% where $\phic$ is given in \eqref{def:phic}. % = 1 - (v_{\infty 2}/v_{\infty 1})^{1/(n_1-n_2)}$.
\end{lemma}

\begin{proof}
With the structure of the flux function \eqref{fluxfunction}
the Rankine-Hugoniot condition \eqref{rhcond} connecting the origin $O$ with any state $\Phi$ % reduces to
takes the form
\begin{align*} %\label{rhkin}
	% \phi_i^+ v_i(\Phi^+) - \phi_i^- v_i(\Phi^-)
	% = s (\phi_i^+ - \phi_i^-) , \quad i=1,2 .
	% \phi_i 
	\sigma(O, \Phi) \phi_i = v_i(\Phi) \phi_i, 
	\quad i = 1,\,2.
\end{align*}
This system of two equations has two possible kinds of solution: 
On the local branches on the axes, $\partial^1$ and $\partial^2$ , 
one of the two equations becomes obsolete, since both sides vanish on the considered axis.
Therefore, the velocity of the remaining equation determines the shock speed as \eqref{speedL}.
On the transversal branches, $\mathcal{C}(\Phis)$ and $\mathcal{C}(\Phimax)$, no such cancellation occurs 
such that for a solution the velocities are required to be equal, giving speed  \eqref{speedT}.
\end{proof}

% \ \\
\noindent
% {\bf Shock speed analysis}.
% The shock speeds along the Hugoniot locus $\mathcal{H}(O)$ can be determined as follows:
%
% The generalization of this lemma to the case $N>2$ is considered in the subsequent sections.
%
% {\color{red} 
%
The triple shock rule as stated in Lemma~\ref{TripleShock} applies to 
the connection of the origin to any state $\Phi \in \mathcal{C}(\Phic)$ having speed \eqref{speedT}
with any middle state $\Phi^M \in C(\Phic)$ having speed \eqref{sv1p}. 
Indeed, $O,\,\Phi,\,\Phi^M$ are not collinear, thus we have $\sigma(O,\,\Phi^M) = \sigma(\Phi^M,\,\Phi) = \sigma(O,\,\Phi)$.
% }
This means that any (shock) solution $O \rightarrow \Phi$ can be constructed by the shock $O \rightarrow \Phi^M$ followed by a second shock $\Phi^M \rightarrow \Phi$ of same speed; both solutions determines the same wave pattern, so the same solution.

Another useful property is the convertibility  of 
the relative velocity $u_i(\Phi)$ with the absolute velocity $v_i(\Phi)$
on the edges:

\begin{lemma}
For states on the edges, $\Phi \in \partial^i, i=1,2$,
one has
\begin{align} \label{vedge}
	v_i(\Phi) = (1-\phi) u_i(\Phi) .
\end{align}
\end{lemma}

\begin{proof}
A state $\Phi \in \partial^i$ on an axis has the representation
$\Phi= \phi \delta_{i1}e_1 + \phi \delta_{i2} e_2$,
where
$e_k, \; k=1,2$, are unit basic vectors
and $\delta_{ik}$ is the Kronecker symbol.

Then, the definition of the relative velocity $v_i(\Phi)$ reduces to
\begin{equation*}
v_i(\Phi) = u_i(\Phi) - \Phi^{\mathrm{T}} \uu = u_i(\Phi) - \phi u_i(\Phi) = (1-\phi) u_i(\Phi) ,\end{equation*}
which establishes the announcement \eqref{vedge}. 
\end{proof}

\section{Characteristic speeds} \label{sec:eig}
System \eqref{systeq} can be written in quasi-linear form as
\begin{align*} %\label{compacteq}
    \Phi_t + \JJ(\Phi) \Phi_x = 0 ,
\end{align*}
where $\JJ(\Phi)$ is the Jacobian matrix of the vector valued flux function
% $\boldsymbol{f}(\Phi) := (f_1(\Phi) f_2(\Phi))^{\mathrm{T}}$. 
$\boldsymbol{f}(\Phi) := (f_1(\Phi), \dots, f_N(\Phi))^{\mathrm{T}}$. 
The structure
of the Jacobian matrix % of the flux function
is examined for general $N$ in~\cite{bbb,bdmv};
for $N=2$ it becomes
\begin{align} \label{jacspec}
\JJ(\Phi) = \begin{pmatrix} \J_{11}(\Phi) & \J_{12}(\Phi) \\ \J_{21}(\Phi) & \J_{22}(\Phi) \end{pmatrix}
	= \begin{pmatrix}
      v_1+u_{11} \phi_1 & u_{12} \phi_1Ê\\
      % & \cdots & u_{1N} \phi_1 \\
      u_{21} \phi_2 & v_2+u_{22} \phi_2Ê\\
%      & \cdots      & u_{2N} \phi_2 \\
%     \vdots & \vdots & & \vdots \\
%    u_{N1} \phi_N & u_{N2} \phi_N & \cdots &
%      v_N+u_{NN} \phi_N
    \end{pmatrix} ,
\end{align}
%\begin{align} \label{jacspec}
%    \JJ(\Phi) =
%    \begin{pmatrix} v_1(\Phi) + \phi_1 v_1'(\Phi) & \phi_1 v_1'(\Phi) & \dots & \phi_1 v_1'(\Phi) \\
%    \phi_2 v_2'(\Phi) & v_2(\Phi) + \phi_2 v_2'(\Phi) & \dots &   \phi_2 v_2'(\Phi) \\
%   \vdots & \vdots & \ddots & \vdots \\
%    \phi_N v_N'(\Phi) &  \phi_N v_N'(\Phi) & \hdots & v_N(\Phi) + \phi_N v_N'(\Phi) .
%\end{pmatrix}
%\end{align}
or, componentwise,
\begin{align*}
    \J_{ij}  =  v_i \delta_{i\smash{j}}
    +
    \phi_i u_{ij},
    \quad
    i,j=1, 2 ,
\end{align*}
where $\delta_{ij}$ is the Kronecker symbol,  $v_i$ is the absolute
velocity \eqref{relvel}, and $u_{ij}$ is specified as
\begin{align} \label{vijcond}
%   v_i = u_i - \bd^{\mathrm{T}} \Phi,
%    \quad i=1, \dots, N; \quad
    u_{ij}  =  u_{ij}(\Phi)  =  u_i'(\Phi) - \Phi^{\mathrm{T}} \bd'(\Phi) - u_j(\Phi),
    \quad i,j=1, 2,
\end{align}
%with the vectors $\Phi=(\phi_1, \dots, \phi_N)^{\mathrm{T}}$,
%    $\bd = (u_1, \dots, u_N)^{\mathrm{T}}$
where
% $\smash{(\bd')} = (u_1'(\phi), \dots,
% u_N'(\phi))^{\mathrm{T}}$
%
%Under the assumption that
% and
%\begin{align}
%\label{upcond}
$\bd'(\Phi) = \Bigl( u_1'(\Phi) \; \; u_2'(\Phi)  \Bigr)^{\mathrm{T}}$.
% $\bd'(\Phi) = (u_1'(\Phi), \dots, u_N'(\Phi))^{\mathrm{T}}$.
%, \quad
%\smash{(\bd')} = (u_1'(\phi), \dots, u_N'(\phi))^{\mathrm{T}}, \quad
%u'_i = \fracd{\partial u_1}{\partial \phi_i} = \dots =
%\fracd{\partial u_n}{\partial \phi_i} .
%   \quad
%   \text{for all}
%   \quad j
%   \in \{ 1, \dots, N \}
%\end{align}
% .
%
% {\it delete? More precisely, the conditions on the hindered settling factors
% that are required by the present analysis, is that the partial derivatives of the
% hindered settling factors with respect to the volume fractions of different
% species coincide.}

Recall that a system is strictly hyperbolic if the Jacobian
matrix of the flux function has distinct real eigenvalues.
%
% \subsection{Eigenvalues % Hyperbolicity 
% for the case $N=2$}
For $N=2$, a system of conservation laws \eqref{systeq}
is strictly hyperbolic if the
discriminant
\begin{align*}
\Delta_{\Phi} := \bigl(\J_{11}(\Phi) - \J_{22}(\Phi)\bigr)^2 - 4
\J_{12} (\Phi) \J_{21} ( \Phi) 
\end{align*}
of the Jacobian matrix \eqref{jacspec} of the flux function $\boldsymbol{f}$ is positive. % definite.
For a hindered settling factor given by \eqref{videf} strict
hyperbolicity holds for identical exponents $n_1=n_2$, see~\cite{bktw}, 
which is proofed by showing algebraically that
$\Delta_{\Phi} > 0$
for $\Phi$ in the interior of the phase space  $\mathcal{D}_{\Phimax}$ for a specification (S3) with $\phimax = 1$; moreover, strict hyperbolicity also holds
for the case % of flux \eqref{n2flux}
with different exponents $n_1 \neq n_2$. This can be shown by a
straightforward calculation \cite{bbb}, which yields
the discriminant composed by a sum of a
square and a positive term
\begin{eqnarray}
\Delta_{\Phi} = [(n_1 \phi_1 - 1) u_1(\Phi) - (n_2 \phi_2 - 1)u_2(\Phi)]^2 + 4 n_1 n_2
\phi_1 \phi_2 u_1(\Phi) u_2(\Phi) % \nonumber \\
   % &=& \bigl(v_{\infty 1} (1-\phi)^{n_1 - 1} ( n_1 \phi_1 - 1 ) - v_{\infty 2} (1-\phi)^{n_2 - 1} (n_2 \phi_2 -1 ) \bigr)^2  
  \label{discrim2} .
\end{eqnarray}
% The fact that for concentrations $\phi_1, \; \phi_2$ inside , 
%
% which becomes 
% \begin{eqnarray}
% \Delta_{\Phi} 
%   &=& \bigl(v_{\infty 1} (1-\phi)^{n_1 - 1} ( n_1 \phi_1 - 1 ) - v_{\infty 2} (1-\phi)^{n_2 - 1} (n_2 \phi_2 -1 )
%  \bigr)^2  \label{discrim3} \\
%   & & \quad + \quad 4 v_{\infty 1} v_{\infty 2}
%  (1-\phi)^{n_1 + n_2 - 2}
%  n_1 n_2 \phi_1 \phi_2 , \nonumber
%  &=&
%  \bigl(u_1(\Phi) ( n_1 \phi_1 - 1 ) - u_2(\Phi) (n_2 \phi_2 -1 )
%  \bigr)^2  \label{discrim2} \\
%   && \quad + \quad 4 n_1 n_2 \phi_1 \phi_2 u_1(\Phi) u_2(\Phi) . \nonumber
%\end{eqnarray}
% when composition of $u_(\Phi)$ as \eqref{uidef} is explicated.
%
This term is positive because of conditions (S1) and (S2) together with 
the bounds $\phi_1, \phi_2 \ge 0, \phi = \phi_1 + \phi_2 \le 1$ given by definition of the invariance domain $\mathcal{D}_{\Phimax}$
and the definition of $u_i(\Phi)$.
The positivity of the discriminant $\Delta_{\Phi}>0$ indicates strict hyperbolicity of the system \eqref{systeq}.
%

% \subsection{Structure of eigensystem for $N=2$}

In the sequel, the structure of the eigensystem 
of the Jacobian matrix $\JJ$ of the flux function~\eqref{fluxfunction}
 is derived. % for the case $N=2$.
% \begin{proof}
The eigenvalue calculation is tedious % ly 
but straightforward, leading to the following Lemma. %, see \eqref{discrim2} \cite{bbb} for an earlier calculation of the discriminant.
% \end{proof}

\begin{lemma}\label{thm:eigen}
The eigenvalues of the Jacobian matrix \eqref{jacspec} % with $N=2$ 
are calculated as
\begin{align*} \begin{split}
\lambda_{1,2} = \fracd{1}{2} \bigl[v_1 + v_2 \bigr] 
	- \fracd{1}{2} \bigl[ n_1 \phi_1 u_1 + n_2\phi_2 u_2 \bigr] \pm \fracd{1}{2} \sqrt {\Delta_{\Phi}},
\end{split}
\end{align*}
where $\Delta_{\Phi}$ is the discriminant % \eqref{discrim2}.
that has the form \eqref{discrim2}.
\end{lemma}

%
% \begin{align*}
%    \Delta_{\Phi} = [u_1 (n_1 \phi_1 - 1) -  u_2 (n_2 \phi_2 - 1)]^2 + 4 u_1 u_2 n_1 n_2 \phi_1 \phi_2
%\end{align*}
%with
%\begin{align}
%  u_i = u_{\infty i} (1-\phi)^{n_i - 1} , \quad i =1,2 .
%\end{align}
%
%
% \subsection{Contact manifold: Main result}
%
% \marginpar{Este teorema sobre contacto doble aparece aislado porque necesita conocimiento sobre los autovalores.}
%

With help of this eigenvalue specification, % analysis,
the characterization of the contact manifold
can be completed 
by the following Theorem, which is formulated for any set $\mathcal{C}(\Phic)$,
such that it applies particularly for the sets $\mathcal{C}(\Phis)$ and $\mathcal{C}(\Phi^\infty)$.

\begin{theorem} \label{theo:eigval}
The eigenvalues for any state $\Phi \in \mathcal{C}(\Phic)$ can be specified as
\begin{align} \label{eq:lambda2onC}
	\lambda_1(\Phic)= v_1(\Phic) - \sqrt{\Delta_{\Phic}} \quad \text{and} \quad
	\lambda_2(\Phi) = v_1(\Phic) .
\end{align}
Therefore, any set % manifold 
$\mathcal{C}(\Phic)$ is a contact manifold with respect to the second characteristic family.
\end{theorem}

\begin{proof}
Since any state $\Phi$ in the set $\mathcal{C}(\Phic)$ % \footnote{Revisar definicion / corolario}
%
% Since the set $\mathcal{C}(\Phic)$ 
is (by definition) characterized by the property
 $v_1(\Phi)=v_2(\Phi)=v_1(\Phic)$,
% we have $(v_1(\Phi)+v_2(\Phi))/2=v_1(\Phic)$
% such that
the eigenvalues % $\lambda_2(\Phi)$
become
\begin{align*}
    \lambda_{1,2}(\Phi) 
    & = v_1(\Phic) - \fracd{1}{2}\bigl[ n_1 \phi_1 u_1(\Phi) + n_2 \phi_2 u_2(\Phi) \bigr]
    \pm \fracd{1}{2} \sqrt{\Delta_{\Phi}} \\
    & = v_1(\Phic) - \fracd{1}{2} \sqrt{\Delta_{\Phi}} \pm \fracd{1}{2} \sqrt{\Delta_{\Phi}}  ,
\end{align*}
giving the pair of eigenvalues \eqref{eq:lambda2onC}.
% Since for any $\Phi \in \mathcal{C}(\Phic)$
%
% Since 
%
% \marginpar{No consigo entender desde la ec. arriba a la ec. abajo en teorema \eqref{theo:eigval}}
%
Here, the last step is justified 
by property \eqref{uvequiv}, {\it i.e.} the equivalence of the equalities $v_1(\Phi)=v_2(\Phi)$ and $u_1(\Phi)=u_2(\Phi)$:
Namely,
one can relate the term in the parenthesis to the discriminant \eqref{discrim2} as
\begin{align*}
    & \bigl[ n_1 \phi_1 u_1  + n_2 \phi_2 u_2 \bigr]^2 % \\
    % \qquad \qquad
    % \quad 
    = [ n_1 \phi_1 u_1  -  n_2 \phi_2 u_2 ]^2 + 4 n_1 n_2 \phi_1 \phi_2 u_1 u_2  
    = \Delta_{\Phi} ,
\end{align*}
which is valid in the special case when $u_1(\Phi)=u_2(\Phi)$.

%
%
%\TD{Esto no es verdad, pero ya lo probaste de otra forma}
%A set where eigenvalues are constant is a contact manifold; therefore the set $\mathcal{C}(\Phic)$ is a contact manifold with respect to the second eigenvalue $\lambda_2$.
%
%

The association of the contact manifold to the second family can be seen by the fact that 
the eigenvalue $\lambda_2$, according to the established values in \eqref{eq:lambda2onC},
 %of the second family (larger eigenvalue) 
is constant on $\mathcal{C}(\Phic)$, $\lambda_2(\Phi) \equiv v_1(\Phic)$ for all $\Phi \in \mathcal{C}(\Phic)$.
%
% \begin{align*}
%   \lambda_2(\Phi^-) = \lambda_2(\Phi^+) =  v_1(\Phic) .
% \end{align*}
% where $s$ is any velocity $v_i(\Phi)$ for $i = 1,\,2$ and $\Phi \in \mathcal{C}(\Phic)$.
With the shock speed established in \eqref{sv1p} in Lemma \ref{lemcont},
for the connection of any two states on $\mathcal{C}(\Phic)$,
one gets
\begin{align*}
    \lambda_2(\Phi^-) = \sigma(\Phi^-, \Phi^+) = \lambda_2(\Phi^+) 
\end{align*}
for all $\Phi^-, \Phi^+ \in \mathcal{C}(\Phic)$ %, which is a double-contact,
establishing that any set $\mathcal{C}(\Phic)$ is a % double-
contact manifold. % $C_{2,2}$.
%Therefore, for $N=2$ any set $\mathcal{C}(\Phic)$ is a contact manifold.
\end{proof}

% (There is a simpler argument for the existence of contact manifold based on the fact that $v_1 = \lambda_2 = v_2$)

Theorem \ref{theo:eigval} applies to both contact manifold $\mathcal{C}(\Phis)$ and $\mathcal{C}(\Phimax)$. %  of the considered model.
For the line $\mathcal{C}(\Phimax)$ the contact manifold is in addition characteristic with respect to the first characteristic family.
% \marginpar{Mejorar}
% Notice that for increasing $\phi_i$ with $\phi_j \equiv 0$, $j \neq i$, the velocity $u_i(\Phi)$ is a decreasing function, 
% such that the eigenvalue $(1 - (1+n_i)\phi_i) u_i(\Phi)$ is also monotonically decreasing. \marginpar{No. Dicha function tiene dos ceros, en el punto $1/(n+1)$ y en el punto $1$}

% For eigenvalues on %  analysis for 
%$\Phi \in \mathcal{C}(\Phis)$ % is given in Eq.~\eqref{eq:lambda2onC}, so 
%we have
%\begin{equation}\label{eigen:cstar}
% \lambda_1(\Phi) = v_i(\Phi) - \sqrt{\Delta_\Phi}\quad\mbox{and}\quad \lambda_2(\Phi) = v_i(\Phi) \equiv v_{\infty i}(1 - \phis)^{n_i},\, i = 1,\,2.
%\end{equation}

% \subsection{Deduce eigenvectors for $N=2$}
% \section{Eigenvectors}
For $\Phi \in \mathcal{C}(\Phis)$ 
the smaller eigenvalue $\lambda_1$ depends on the discriminant
and only the bigger eigenvalue $\lambda_2(\Phi) \equiv v_1(\Phis)$ is constant, 
Thus, a contact manifold $\mathcal{C}(\Phic)$ is an integral curve of the second family.
On the maximum packing manifold for $\Phi \in \partial^{\infty}=\mathcal{C}(\Phimax)$, where $\phi_1 + \phi_2 = 1$ holds, both eigenvalues vanish:
$\lambda_1(\Phi) = \lambda_2(\Phi) = 0$.
% {\color{red}
% And both characteristic directions coincide with the $\partial^\infty$ boundary. % corresponder'a donde se considered los vectores propios
% }

% \ \\
% \noindent
% {\bf Eigenvalues on axes}.
% \subsection{Eigenvalues on the remaining parts of the Hugoniot locus of the origin}

Theorem \ref{theo:eigval} states explicit expressions for the eigenvalues on the contact manifold $\mathcal{C}(\Phis)$ and $\mathcal{C}(\Phimax)$, 
which form part of the Hugoniot locus of the origin as identified in Lemma \ref{hlorigin2}.
The remaining eigenvalues along the Hugoniot locus of the origin
can be evaluated directly from the general eigenvalues according to 
Lemma~\ref{thm:eigen}.
For $\Phi$ on an edge $\partial^i, \; i = 1,\,2$, {\it i.e.}, $\phi_i = 0$ for some $i = 1,\,2$, as specified in \eqref{def:boundaries},
 the discriminant reduces to 
  $\Delta_\Phi = [ u_{3-i} (\Phi)(n_{3-i} \phi_{3-i} - 1) ]^2$, 
% $\Delta_\Phi = [ u_i(\Phi)(n_i \phi_i - 1) ]^2$, 
% $\Delta_\Phi = [u_1(\Phi)(n_1\phi_1 - 1) - u_2(\Phi)(n_2\phi_2 - 1)]^2$, 
%
where the subindex $3-i$ becomes $2$ on the axis $\partial^1$ and $1$ on the axis $\partial^2$.
A simple calculation shows that for $\Phi = (\phi,\,0)^T \in \partial^1$ the eigenvalues are
\begin{eqnarray*}
\lambda_{1,\,2}(\Phi) &\,\,\,=\,\,\,& \frac{1}{2} \Big\{ \big[(1 \mp 1) - (2 + (1 \mp 1)n_1 )\phi \big]u_1(\Phi) + (1 \pm 1)u_2(\Phi) \Big\} \\
% \label{lam12:phi1}
                &\,\,\,=\,\,\,& \begin{cases}
                     % \lambda_a(\Phi) \,\,:=\,\, 
                     u_2(\Phi) - \phi u_1(\Phi),  \\ %& \!\!\!\!\! = v_{\infty 2}(1- \phi)^{n_2-1} - v_{\infty 1}\phi(1 - \phi)^{n_1-1} \\
                     % \lambda_b(\Phi) \,\,:=\,\, 
                     (1 - (1+n_1)\phi) u_1(\Phi),    %& \!\!\!\!\! = v_{\infty 1}(1 - (1+n_1)\phi)(1 - \phi)^{n_1-1}
                    \end{cases}
\!\!\!\!\!\!
\end{eqnarray*}
and for $\Phi = (0,\,\phi)^T \in \partial^2$ the eigenvalues are
\begin{eqnarray*}
\lambda_{1,\,2}(\Phi) &\,\,\,=\,\,\,& \frac{1}{2} \Big\{ (1 \mp 1)u_1(\Phi) + \big[(1 \pm 1) - (2 + (1 \pm 1)n_2 )\phi \big]u_2(\Phi) \Big\} \\
% \label{lam12:phi2}
                &\,\,\,=\,\,\,& \begin{cases}
                     % \lambda_a(\Phi) \,\,:=\,\, 
                     (1 - (1+n_2)\phi) u_2(\Phi), \\ %& \!\!\!\!\! = v_{\infty 2}(1 - (1+n_2)\phi)(1 - \phi)^{n_2-1} \\
                     % \lambda_b(\Phi) \,\,:=\,\, 
                     u_1(\Phi) - \phi u_2(\Phi) .    %& \!\!\!\!\! = v_{\infty 1}(1- \phi)^{n_1-1} - v_{\infty 2}\phi(1 - \phi)^{n_2-1}
                    \end{cases}                   
\!\!\!\!\!\!
\end{eqnarray*}
%
% \TD{POSIBLEMENTE PODEMOS REDEFINIR LOS AUTOVALORES DE SIGUIENTE MANERA:
% The reason is that in the inflection points on the axes one can set $\lambda_1(\Phi) =  \lambda_a(\Phi)$.}
%
Within these eigenvalue characterizations % expressed in \eqref{lam12:phi1}, \eqref{lam12:phi2} 
we can distinguish the two types % the eigenvalues
% $$\lambda_a(\Phi) := \big(1 - (1+n_i)\phi\big) u_i(\Phi),\quad\mbox{for}\quad \Phi \in \partial_i,\,\qquad \phi = \phi_i,\,i =1,\,2 , $$
\begin{eqnarray}
\label{lam12:phiab}
                &\,\,\,  \,\,\,& \begin{cases}
                     \lambda_a(\Phi) \,\,:=\,\, (1 - (1+n_i)\phi) u_i(\Phi), & \quad\mbox{for}\quad \Phi \in \partial_i, \; i=1,2, \\ %& \!\!\!\!\! = v_{\infty 2}(1 - (1+n_2)\phi)(1 - \phi)^{n_2-1} \\
                     \lambda_b(\Phi) \,\,:=\,\, u_{3-i}(\Phi) - \phi u_i(\Phi), & \quad\mbox{for}\quad \Phi \in \partial_i, \; i=1,2 ,    %& \!\!\!\!\! = v_{\infty 1}(1- \phi)^{n_1-1} - v_{\infty 2}\phi(1 - \phi)^{n_2-1}
                    \end{cases}                   
\!\!\!\!\!\!
\end{eqnarray}
where the subindex $3-i$ denominates the complementary index.
% $j$ denominates the complementary index such that $i+j=3$ holds. 
%
%
The notation $\lambda_a(\Phi), \lambda_b(\Phi)$
with new subindices instead of $\lambda_1(\Phi), \lambda_2(\Phi)$
is used since the order
\begin{align} \label{lambdaorder}
	\lambda_a(\Phi) = \lambda_1(\Phi)  \le \lambda_2(\Phi) = \lambda_b(\Phi) 
\end{align}
cannot be always guaranteed.
However, the notation $\lambda_1$, $\lambda_2$ is used whenever 
the order % $\lambda_1 < \lambda_2$ 
\eqref{lambdaorder}
can be assured,
{\it e.g.} at the origin, where
\begin{equation*}
	\lambda_1(O) = \lambda_a(O) = v_{\infty 2} <  v_{\infty 1} = \lambda_b(O)  = \lambda_2(O) .
\end{equation*}
% holds. % Thus, we have described all the eigenvalues along $\mathcal{H}(O)$.
%\marginpar{Lo tenemos? Mas adelante}
%

%
Generally, order changes can occur on both axes, as can be seen from
Fig.~\ref{fig:edgeanalysis}~(b). 
The order \eqref{lambdaorder} holds only in the absence of such order changes.
However, the occurrence of coincidence points does not affect the subsequent classification.
In the sequel, it is shown how the eigenvalues $\lambda_a(\Phi)$ and $\lambda_b(\Phi)$ 
assume the role of the eigenvalue of the first or the second family
in dependence of the value $\phi$.

\subsection{Inflection curves}

The highly non-linear structure of the flux function 
% sternly impedes 
makes it difficult, if not impossible to obtain an explicit formula for inflection manifolds \cite{Cido}. 
% \marginpar{references: \cite{Cido}}
However, a characterization for states on the coordinate axes of the phase space can be obtained,
since the inflection points there correspond to those of scalar equations. 
Namely, the eigenvalue $\lambda_a(\Phi)$ has the same form as the first derivative of the scalar flux function ($N=1$),
whereas there is no correspondence for the additional eigenvalue $\lambda_b(\Phi)$, which emerges for the $2 \times 2-$system ($N=2$). % and is characterized by \eqref{lam12:phiab}.
% only eigenvalue that occurs in the scalar case $N=1$.
%
% A direct calculation shows that 
%
% MAPLE
% L := (1-(1+n)*p)*vinf*(1-p)^(n-1);
% dL := diff(L, p);
% factor(dL);
%
%
The derivative
\begin{align} \label{lambdader}
	\lambda_a'(\phi) = n_i \bigl((1+n_i)\phi - 2\bigr)u_i(\Phi)/(1-\phi), \quad i=1,2
\end{align}
% Notice that $\lambda'(\phi)$ 
is well defined for $\phi \in [0,\,1)$.
From \eqref{lambdader}
we have that $\lambda_a'(\phi)$ may vanish at $\phi = 1$ and at $\phiinf = 2/(1+n_i)$; the first case is less relevant because it represents a state on a vertex of the phase space. 
We have that $\phiinf \in (0,\,1)$ since $n_i$ is larger than one due to condition (S2).
Thus the state $\phiinf$ represents an inflection point
and the states $\Phiinf_i$ such that $\phi_i = \phiinf$ and $\phi_j = 0$ for $i \neq j \in  \{�1,2 \}$ 
belong to an inflection manifold.
The pertinence of the inflection points to the first characteristic family can be shown for general parameter settings.

\begin{lemma}\label{cla:inflection}
There is at least one inflection point on each axis of the phase space \eqref{phasespace},
having locations $\Phiinf_1 := \bigl( 2/(1+n_1),\,0 \bigr)^T$ and $\Phiinf_2 := \bigl( 0,\,2/(1+n_2) \bigr)^T$.
Moreover, the corresponding eigenvalues belong to the first characteristic family.
\end{lemma}
%
%% \marginpar{Change subscript of $\Phi_m, \Phiinf$ ?} OK!!

\begin{proof}
The location of the inflection points on the axes is obtained by finding the zeros of the corresponding eigenvalue derivatives \eqref{lambdader}.

The attribution of the inflection points to a certain characteristic family is done by comparing the magnitudes 
of the eigenvalues $\lambda_a(\Phiinf_i)$ and $\lambda_b(\Phiinf_i)$, % =u_j(\Phiinf_i) - \phiinf u_i(\Phiinf_i)$, 
both given by \eqref{lam12:phiab}:
Substituting $\phiinf = 2/(1+n_i)$
one observes that 
$$\lambda_a(\Phiinf_i) + \phiinf u_i(\Phiinf_i) = (1 - n_i \phiinf) u_i(\Phiinf_i) = \frac{1 - n_i}{1 + n_i} u_i(\Phiinf_i). $$
Since $u_i$, $u_j$ are nonnegative and $(1 - n_i)/(1 + n_i)$ is always negative
by the parameter setting (S2), we have that 
\begin{align*}
	\lambda_a(\Phiinf_i) < u_j(\Phiinf_i) - \phiinf u_i(\Phiinf_i) = 	\lambda_b(\Phiinf_i) .
\end{align*}
Thus, in the inflection points on the axes, the eigenvalue $\lambda_a$ corresponds to the first (smaller) eigenvalue $\lambda_1$
such that % (For the inflection points on the axes 
the eigenvalue order \eqref{lambdaorder} holds.
\end{proof}

\subsection{Eigenvectors } % and integral curves}

% We consider the characteristic polynomial of
The Jacobian matrix \eqref{jacspec} with components\ \eqref{vijcond} can be written as 
a rank two modification % \TD{Hacer directamente para el caso $N=2$}
% of a diagonal matrix
\begin{align} \label{ABC}
\JJ = D + \sum_{k=1}^2 a_k b_k^{\mathrm{T}} = D+BA^{\mathrm{T}}
\end{align}
of the diagonal matrix $D=\mathrm {diag}(v_1, v_2) \in
\mathbb{R}^{2 \times 2}$,
where
\begin{align} \label{ABdef}
% D=\begin{pmatrix}
%    v_1 & 0 & \hdots & 0 \\
%    0   & v_2 & \ddots  & \vdots\\
%    \vdots & \ddots & \ddots & 0 \\
%    0 & \hdots & 0 & v_N \\
%  \end{pmatrix},
%  \quad
%
  B=\begin{pmatrix}
    \phi_1 u_1' (\Phi) & - \phi_1 \\
    \phi_2 u_2' (\Phi) & - \phi_2   
  \end{pmatrix},
  \quad
  A=\begin{pmatrix}
    1 & u_1 +  \Phi^{\mathrm{T}} \uu'(\Phi)  \\
    1 & u_2 + \Phi^{\mathrm{T}} \uu'(\Phi) 
  \end{pmatrix} .
\end{align}
%
%
% \eqref{ABC} for which
In a rank two modification one 
introduces column
vectors $a_k, b_k \in \mathbb{R}^2, \; k \in \{ 1, 2 \}$, in % the sum 
%where the product $B A^{\mathrm{T}}$ of the
matrices $A, B \in \mathbb{R}^{2 \times 2}$ which both have rank two.
%
%
% For arbitrary $N$, the general 
The formula for the components of  a right eigenvector 
can be deduced from secular equations and is
given as a function of eigenvalues \cite{bdmv}
\begin{align} \label{rij}
    r_{ij}(\lambda) = \fracd{1}{v_j - \lambda} \biggl[ b_j^1 \sum_{k=1}^2
\fracd{a_k^1 b_k^2}{v_k - \lambda}
    - b_j^2 \biggl( 1 + \sum_{k=1}^N \fracd{a_k^1 b_k^1}{v_k - \lambda}
    \biggr)
      \bigg] ,
\end{align}
% where the variables are components of the interpretation of the
% Jacobian matrix as rank-2 modification of the diagonal matrix
% consisting of the absolute velocities
% and the modifying matrices consist of 
where the parameters
%\begin{align} \label{r2param}
%	b_{j,1}=\phi_j u_j'(\Phi), \quad
%	b_{j,2}=-\phi_j, \quad
%	a_{j,1}=1, \quad
	% \Biggl[  a_{i,2}=\Phi^{\mathrm{T}} \uu'_i(\Phi) + u_i
	% \quad \text{(does not apply for right eigenvector)} \Biggr]
% \end{align}
% such that
\begin{align} \label{abj}
	a_j^1= 1, \quad a_j^2 = u_j(\Phi) + \Phi^{\mathrm{T}} \uu', \quad
	b_j^1 = \phi_j u_j', \quad b_j^2 = - \phi_j,
\end{align}
correspond to the entries of the matrices $A$ and $B$ in \eqref{ABdef}.
Whereas formula \eqref{rij} holds for general rank two modifications of form \eqref{ABC},
which are valid for system of arbitrary size, the following Lemma breaks it down for the considered model with $N=2$.

% \marginpar{Display somewhere this rank-2 modification}

% In the sequel an analytic statement of the eigenvectors are derived.

\begin{lemma}% [Eigenvectors for $N=2$]
% For $N=2$, 
The right eigenvectors % \marginpar{right?? La verdad es que no entiendo nada. Que es $\Lambda$?} % $r_i(\Phi)$ 
of the Jacobian matrix \eqref{jacspec}  % with $N=2$
are functions of the corresponding 
eigenvalues $\lambda_i(\Phi), i=1,2$, given as
\begin{align} \label{eigvector}
    r(\Phi, \lambda_i(\Phi)) = \begin{pmatrix}
        \phi_1 (v_2(\Phi) - \lambda_i(\Phi)) + \phi_1 \phi_2 (u_2'(\Phi) - u_1'(\Phi)) \\
        \phi_2 (v_1(\Phi) - \lambda_i(\Phi)) - \phi_1 \phi_2 (u_2'(\Phi) - u_1'(\Phi))
    \end{pmatrix}. % , \quad i=1,2   .
\end{align}
\end{lemma}

\begin{proof}
By the abbreviation
\begin{align*}
	\Lambda_j % = \Lambda(\Phi, \lambda(\Phi)) 
	= \fracd{\phi_j}{v_j - \lambda} ,
\end{align*}
% \TD{Del lado dericeho es $\Lambda_j$ tambien?}
one can rewrite the eigenvector \eqref{rij} with parameters \eqref{abj} compactly as
\begin{align*}
	r_{*j}(\Phi, \lambda(\Phi)) = \Lambda_j % (\lambda) 
	\Biggl[�1+ \sum_{k=1}^2 \Lambda_k % (\lambda)
	 (u_k' -u_j') \Biggr] , \quad j = 1, 2 ,
	% \fracd{1}{v_i-\lambda} \Biggl[ \phi_i u_i \sum_{k=1}^N \fracd{\phi_k}{} \Biggr]	
\end{align*}
% These eigenvectors
% which reduces for $N=2$ to
with the components
\begin{align*}
	r_{*1} = % (\Phi) = 
	% \begin{pmatrix}	
	\Lambda_1% (\lambda) 
	\bigl[�1 + \Lambda_2 % (\lambda) 
	(u_2' - u_1') \bigr]
	 	% , \quad
	= \fracd{\phi_2}{v_2 - \lambda} + \fracd{\phi_1 \phi_2 (u_1' - u_2')}{(v_1 - \lambda)(v_2 - \lambda)}, \\
	r_{*2} = % (\lambda) = 
		\Lambda_2 % (\lambda) 
		\bigl[�1 + \Lambda_1 % (\lambda) 
		(u_1' - u_2') \bigr]
	% \end{pmatrix}
	= 
	% \begin{pmatrix}	
	\fracd{\phi_1}{v_1 - \lambda} + \fracd{\phi_1 \phi_2 (u_2' - u_1')}{(v_1 - \lambda) (v_2 - \lambda)} .
	% \end{pmatrix}	
\end{align*}
Multiplying by $(v_1-\lambda)(v_2 - \lambda)$ the eigenvectors
get the form \eqref{eigvector}.
\end{proof}

% \subsection{Exploiting structure of Jacobian matrix} % Eigenvectors on $\mathcal{C}(\Phic)$ and in $O$}

Special eigenvectors, in particular those for the values of the Hugoniot locus of the origin, 
can be obtained from either exploiting the structure  of the Jacobian matrix \eqref{jacspec} or from using the analytical form of the eigenvectors \eqref{eigvector}.
%
% \begin{lemma}
% For a $2\times 2$ system
%
For instance, the eigenvectors on a set $\mathcal{C}(\Phic)$ that correspond to the second eigenvalue
$\lambda_2(\Phic) = v_1(\Phic)$ have the form
\begin{align} \label{ev2}
	r_2(\Phic) = \begin{pmatrix} 1 \\ -1 \end{pmatrix}.
\end{align}
% On the axis $\partial^1$ both eigenvectors have the form $r_1(\Phi)=r_2(\Phi)=(1, \; 0)^{\mathrm{T}}$. % \marginpar{No entendi esto. Un vector deberia ser transversal (?)}
% On the axis $\partial^2$ both eigenvectors have the form $r_1(\Phi)=r_2(\Phi)=(0, \; 1)^{\mathrm{T}}$.
In the origin one has $r_1(O)=(0, \; 1)^{\mathrm{T}}$ and $r_2(O)=(1, \; 0)^{\mathrm{T}}$. 
% \end{lemma}
%

% ((Parte del texto del texto despuŽs del teorema 1))
% {\color{red}
% \begin{corollary}\label{cor:contact}
% A contact manifold $\mathcal{C}(\Phic)$ is an integral curve of the second family.
% \end{corollary}

% \begin{remark}
% Another proof of Lemma~\ref{lemcont} follows from Corollary~\ref{cor:contact}. As $\mathcal{C}(\Phic)$ is an invariant manifold
% (in the sense that solutions that start on the invariant manifold remain there for all time), Theorem 2 in \cite{Tem83} guarantees that $\mathcal{C}(\Phic)$ is a contact manifold. (Which turns out to be of second family.)
% \end{remark}
% }

% \subsection{Integral curves}  of the $2 \times 2$ system}
\subsection{Illustration of a benchmark example}

For illustration, a benchmark example ({\bf Example 1}) is considered 
with parameter setting % \eqref{parametersetting1}
\begin{align} \label{parametersetting1}
v_{\infty 1} = 1,\, v_{\infty 2} = 1/2,\qquad and %, $v_{\infty 3} = 1/4$,
\qquad n_1 = 4,\, n_2 = 3, % , $n_3 = 2$
\end{align}
that satisfies specifications (S1)--(S3).
% the inflection curve is shown in Fig.~\ref{fig:curves2x2}.
%
The integral curves % of this benchmark example 
in the $\phi_1 \phi_2-$coordinate plane are
shown in Fig.~\ref{fig:curves2x2} (left). % for the benchmark example with parameters \eqref{parametersetting1}.
%
% From Fig.~\ref{fig:curves2x2} 
The arrows point into the direction of increasing eigenvalues.
The first characteristic family is crossing all lines of constant $\phi$ % (blue), 
and the second family is connecting the axes. 
It can be recognized that 
the second family is genuinely nonlinear,
whereas genuine nonlinearity 
of the first family 
is lost at an inflection manifold,
where the corresponding eigenvalues take their minimum.
 % (solid line).
%
% \begin{align}
%	\nabla \lambda_1(\Phi) \cdot r_1(\Phi) = 0 .
% \end{align}
% The rarefaction curves of  the first characteristic family origin in the inflection manifold
% at which the corresponding eigenvalues take their minimum.

The direction of increasing eigenvalues of the second % and the third
characteristic family switches at the contact manifold.
%
% The corresponding $2\times 2$ system still has a integral curve
% direction switch 
%
% the interior contact manifold $\mathcal{C}(\Phis)$ (solid line).
%
% {\bf semi-coincidence plane}. 
% {\bf Needs definition of semi-coincidence plane, or alternative concept.}
This direction
switch impacts on the solution structure of the Riemann problem
$\RP(O,\Phi^+)$, depending on which side of the contact
manifold the right state $\Phi^+$ is positioned, see Fig.~\ref{fig:curves2x2} (right) 
for results of simulations by a finite difference method. Only a slight change of the Riemann data
provokes a fundamental change of the solution path in the phase space.

\begin{figure}[ht]
\begin{center} % ShockClassification.png
\begin{tabular}{c}
% (a) & (b) \\
\includegraphics[width=0.49\textwidth]{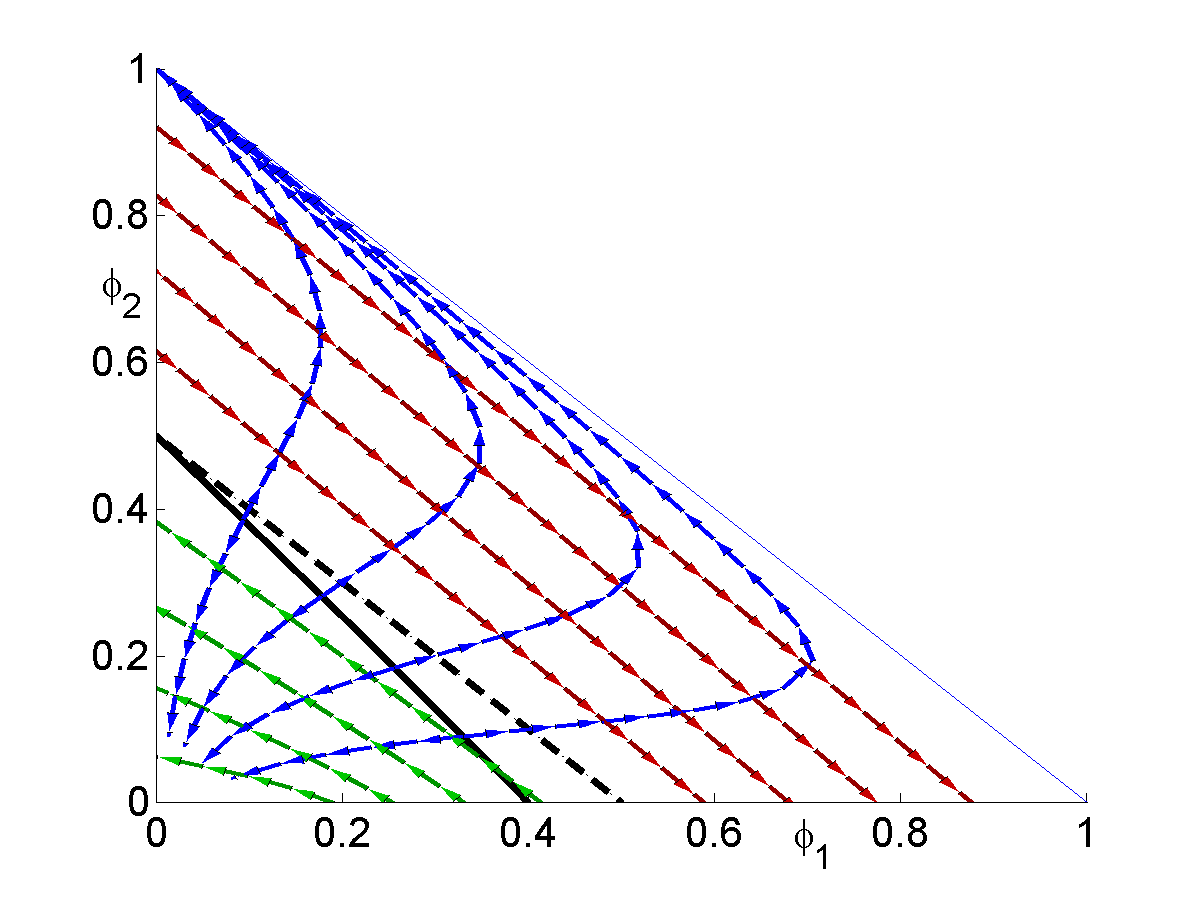}
% &
\includegraphics[width=0.49\textwidth]{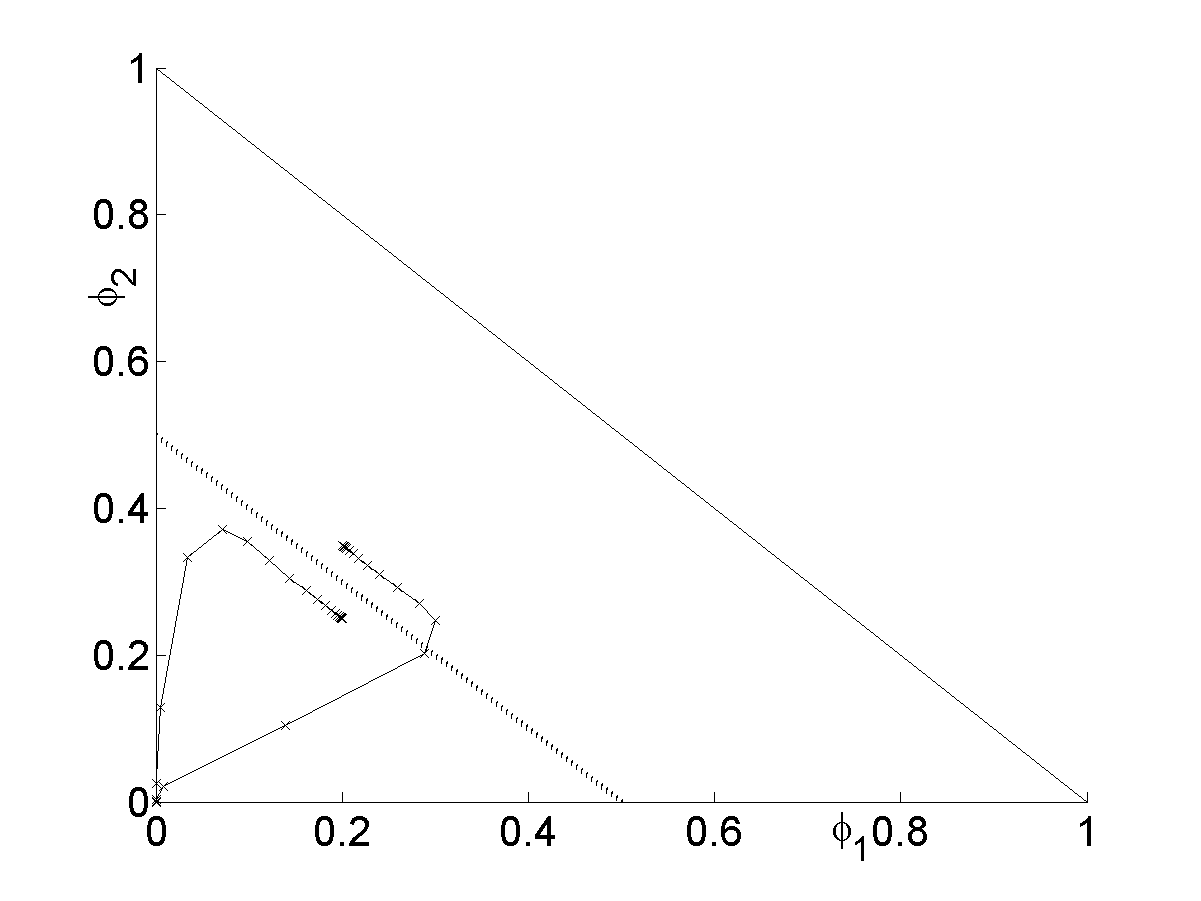}
\end{tabular}
\end{center}
\caption{Phase space of Example 1 with parameters~\eqref{parametersetting1}.
Left: Integral curves in the $\phi_1 \phi_2-$coordinate plane:
First family (crossing all lines with constant $\phi$) % (blue), 
and second family (connecting the axis). % (red)
% third family (green). 
The integral curves point into the direction of increasing eigenvalues;
contact manifold (dashed), inflection manifold of the first family (solid).
Right: Solution of Riemann problems $RP(O,(0.2,0.25))$,
$\RP(O,(0.2,0.35))$ by a finite difference method (low resolution). } % caption
\label{fig:curves2x2}
\end{figure}

%
% \marginpar{reemplazar figura por otra? Creo que no hace falta}
%
% However, as we will see in Section, the solution profile changes smoothly.

% \subsection{Algebraic properties of eigensystem}
% \TDD{Show some properties of eigenvalues and eigenvectors
% algebraically, e.g. that for $\Phi^-,\Phi^+$ on the C-plane one has
% $\lambda(\Phi^-)=\lambda(\Phi^+)$.}

% fin
% eigenvalue analysis 
%

\section{Quasi-umbilic point} \label{subs:quasiumbilic}

% \marginpar{Eventually revise exposition}

The best examined cases of the loss of hyperbolicity are related to an ``umbilic point'' \cite{impt88, Matos, ss87},
which is % usually - Àhay excepciones de lo usual?
an isolated point where strict hyperbolicity fails.
A description of new types of isolated points with loss of hyperbolicity is given in \cite{pantersHYP}. 
In particular, if there is a connected set of points with loss of hyperbolicity, then more refined characterizations are needed. For instance, a generalization of the umbilic point is the coincidence point. 

\begin{definition} \label{def:umbilic}
A point $\Phi$ is called a {\bf coincidence point} 
of the PDE~(\ref{systeq}) with flux function
$\boldsymbol{f} (\Phi)$ 
% for the flux
% $\boldsymbol{f} ( \Phi)$ 
if the eigenvalues of the
Jacobian matrix $\JJ(\Phi)$  of the flux function % $d\boldsymbol{f}$
coincide at this point.
We say that a coincidence point $\Phi^*$ is an {\bf umbilic point}
if it satisfies the following conditions:

(H1) The Jacobian matrix $\JJ(\Phi^*)$ % $d\boldsymbol{f}(\Phic)$
is diagonalizable.

(H2) There is a neighborhood $V$ of $\Phic$ such that $\JJ(\Phi)$
has distinct eigenvalues for all $\Phi \in
V\backslash\{\Phi^*\}$. % \\
\end{definition}
Typically, we would expect that umbilic points are isolated
coincidence points. However, the previous definition of an umbilic point may fail in either of the two conditions. Following \cite{panters}, we
classify an isolated coincidence point where condition (H1) fails as a {\bf quasi-umbilic point}. It seems that the
classification of quasi-umbilic points
appears at the first time
% {\color{red} 
in \cite{panters}, in \cite{pantersHYP} a detailed description is given. % }

%Following Definition.... \TD{Aqui va la definicion de quasi-umbilic y diagonalizable manifold // Esta prueba tiene que estar antes de Thm.~\ref{lem:classification}. - Hay dudas que todo es correcto. La ecuacion $R_1(\phi)$ no es una funcion. No tiene un cero en $\phi = 1/n_1$.}

\begin{lemma}\label{lem:quasiumbilic}
There is a unique % "a unique" es correcto: http://wordwhirled.blogspot.com/2005/12/unique-or-unique.html
coincidence point on $\Phi \in \partial^1$. Upon the choice of parameters there may exist up to two coincidence points on $\Phi \in \partial^2$.
%These coincidence points turn out to be quasi-umbilic.
\end{lemma}

% \marginpar{In the proof it is not shown that the coincidence points are quasi-umbilic.}

\begin{proof}
By definition, a state $\Phi$
is a coincidence point if and only if $\lambda_a(\Phi) = \lambda_b(\Phi)$ %, given in \eqref{lam12:phiab}, 
holds.
Namely, from the eigenvalue expression \eqref{lam12:phiab}, one has a coincident point on axis $\Phi \in \partial^1$ 
% {\it i.e.}, 
if there exists some $\phi_1 \in [0,\,1]$ such that $(1 - n_1\phi_1)u_1(\Phi) = u_2(\Phi)$ holds. 
Therefore, we must look for such a value $\phi_1$ which is a zero root of $R_1(\phi) - v_{\infty 2}/v_{\infty 1} = 0$, where
\begin{align*}
	R_1(\phi) := (1 - n_1 \phi) (1 - \phi)^{n_1-n_2}.
\end{align*}
Notice that by (S1) and (S2) the ratio $v_{\infty 2}/v_{\infty 1}$ is smaller than one and that the power $n_1 - n_2$ is positive.
%
% not needed for proof:
% The function $R_1(\phi)$ is positive for $\phi \in [0,\, 1/n_1)$ and negative for $\phi \in (1/n_1,\,1)$. % ; it vanishes at $\phi \in \{ 1/n_1,\,1\}$. 
%
Since $R_1(0) = 1$, $R_1(1/n_1)=0$ and $R_1'(\phi) = -n_1(1 - \phi)^{n_1-n_2} - (n_1 - n_2)(1 - n_1\phi)(1 - \phi)^{n_1-n_2-1}$ is negative for all $\phi \in [0,\,1/n_1]$, we conclude that $R_1(\phi) = v_{\infty 2}/v_{\infty 1}$ occurs at a single point, say $\phiu$, which is the unique zero root of $R_1(\phi) - v_{\infty 2}/v_{\infty 1} = 0$. Let us denote such a state as $Q_1 := (\phiu,\,0)^T$.

Similarly, for the axis $\Phi \in \partial^2$, from the eigenvalue expression \eqref{lam12:phiab}, we look for a value $\phi_2 \in [0,\,1]$ 
(while $\phi_1=0$)
which is a zero root of $R_2(\phi) - v_{\infty 1}/v_{\infty 2} = 0$, where
$$R_2(\phi) := (1 - n_2 \phi) (1 - \phi)^{n_2-n_1},\qquad \phi \neq 1.$$
Notice that $R_2(0) = 1$ and $R_2(1/n_2) = 0$ hold, and since $n_1 > n_2 > 1$, in the limit $\lim_{\phi\rightarrow 1-}R_2(\phi) = -\infty$ holds. Moreover, $R_2(\phi)$ is positive for $\phi \in [0,\, 1/n_2)$ and negative for $\phi \in (1/n_2,\,1)$. As the ratio $v_{\infty 1}/v_{\infty 2}$ is larger than one, there is no coincidence point if $R_2(\phi)$ is always smaller than such a ratio. For $\phi \in [0,\,1)$ we have
\begin{eqnarray*}
R_2'(\phi) &=& -n_2(1 - \phi)^{n_2-n_1} - (n_1 - n_2) (1 - n_2\phi)(1 - \phi)^{n_2-n_1-1} \\
           &=& [- n_2 (1-\phi) + (n_2 - n_1) (1 - n_2\phi)](1 - \phi)^{n_2-n_1-1}. 
\end{eqnarray*}
Thus, a single extremum of $R(\phi)$ occurs at 
\begin{align} \label{phim}
	\phi^m := \fracd{n_1 - 2n_2}{(n_1 - n_2 - 1)n_2}	
	% (n_1 - 2n_2)/\big((n_1 - n_2 - 1)n_2\big).
\end{align} 
There is one single coincidence point on the axis if $R_2(\phi^m) = v_{\infty 1}/v_{\infty 2}$.
The existence of two coincidence points occurs if $R_2(\phi^m)$ is larger than $v_{\infty 1}/v_{\infty 2}$; say $Q_2 := (0,\,\phiu_2)^{\mathrm{T}}$,  $Q_3 := (0,\,\phiu_3)^{\mathrm{T}}$ where $R_2(\phiu_2) = R_2(\phiu_3) = v_{\infty 1}/v_{\infty 2}$.
\end{proof}

%\marginpar{Definition of $V$?}
% \marginpar{Why consider only $n_1-n_2 \in \mathbbm{N}$}
% \marginpar{Los calculos son faciles, si no las cuentas son horribles}
The necessary and sufficient conditions for coincidence points on $\partial^2$ are difficult to evaluate. 
% {\color{red}
Let us define the ratio $W := v_{\infty 1}/v_{\infty 2}$. On the one hand we know that $V$ exceeds one due to condition (S1). % , but we do not have control about its grow. -- ÀÀcomo podr'a ser que una constante crece??
% }
On the other hand the powers $n_1$, $n_2$ are also free parameters. 
Some calculations are possible if the difference between $n_1$ and $n_2$ is a natural number.
% (a restriction that facilitates some calculations), 
For example, 
for $n_1 - n_2 = 1$, there are two coincidence points if $n_2 > W$; 
for $n_1 - n_2 = 2$, there are two coincidence points if $n_2 > 2(W + \sqrt{W^2 - W})$; 
for $n_1 - n_2 = 3$, there are two coincidence points if $4n_2^3 - 27 W (n_2 - 1)^2 > 0$.
%
% As we may notice, the conditions change within the difference $n_1 - n_2$ and the ratio $V$. 
% Nonetheless, it is easy to show that
The inequality $R_2(\phi_m) < 1$ holds for powers satisfying $n_2+1 < n_1 <2n_2$, 
and for $\phi_m$ definied in \eqref{phim}. 
In this case there are not coincidences on the axis $\partial^2$.

% (An algebraic calculus shows that $\phiu_{2,3} < 1/n_2$ holds.)

For $\Phi=(\phi,0)^{\mathrm{T}}$ the Jacobian matrix is an upper triangular matrix
and for $\Phi=(0,\phi)^{\mathrm{T}}$ a lower triangular matrix. In both cases the Jacobian matrix is diagonalizable.
In view of Definition \eqref{def:umbilic}, condition (H2) is satisfied, but (H1) not necessarily.
Given that the diagonalization is valid for all $\Phi \in \partial^1$, the coincidence point
$Q_1 = (\phiu,0)^{\mathrm{T}}$ is a quasi-umbilic point,
where the Jacobian matrix has the form
\begin{align*} % \label{jac:umbilic}
	 \begin{pmatrix} c_{\sigma} & *\\ 0 & c_{\sigma} \end{pmatrix}
\end{align*}
is a Jordan block with $* \neq 0$ and therefore not diagonalizable.
This means that (H1) is violated, but (H2) is satisfied.
% {\color{red}
Therefore, from Lemma~\ref{lem:quasiumbilic} we have the following result.

% \TD{Does $* \neq 0$ always hold? - If should be either part of proof or not a part of the lemma.}

\begin{corollary}
Any isolated coincidence point on the boundary turns out to be quasi-umbilic.
\end{corollary}
% }

% On the contact manifold $\mathcal{C}(\Phimax)$ all states are coincidence points since both eigenvectors vanish in all points of the line. % ((no parece correcto))

% Since there is the order switch
% \begin{align}
% 	\lambda_a(O) = u_2(O) < u_1(O) = \lambda_b(O), 
% 	\qquad
% 	\lambda_b(\Phis) > \lambda_b(\Phis), 
% \end{align}
% there is a state $\Phi^{u}=(\phiu,0)^{\mathrm{T}} \in \partial^1$,
% %\marginpar{$\phiu$ es el $\phi_c$ de la fig.\ref{fig:edgeanalysis}, es una raiz unica}
% $\phiu \in (0, \phis)$ such that
% \begin{align}
% 	\lambda_{u} :=\lambda_a(\Phiu) = \lambda_b(\Phiu)
% \end{align}
% 
% This state is isolated since on the axis the derivative 
% $\partial_{\phi} ( \lambda_a(\Phi) - \lambda_b(\Phi) )$
% with respect to $\phi$ is non-zero, and in the interior of $\mathcal{D}_{\Phimax}$
% the discriminant is positive.

\begin{figure}[htb]
\centering
\includegraphics[width = \textwidth]{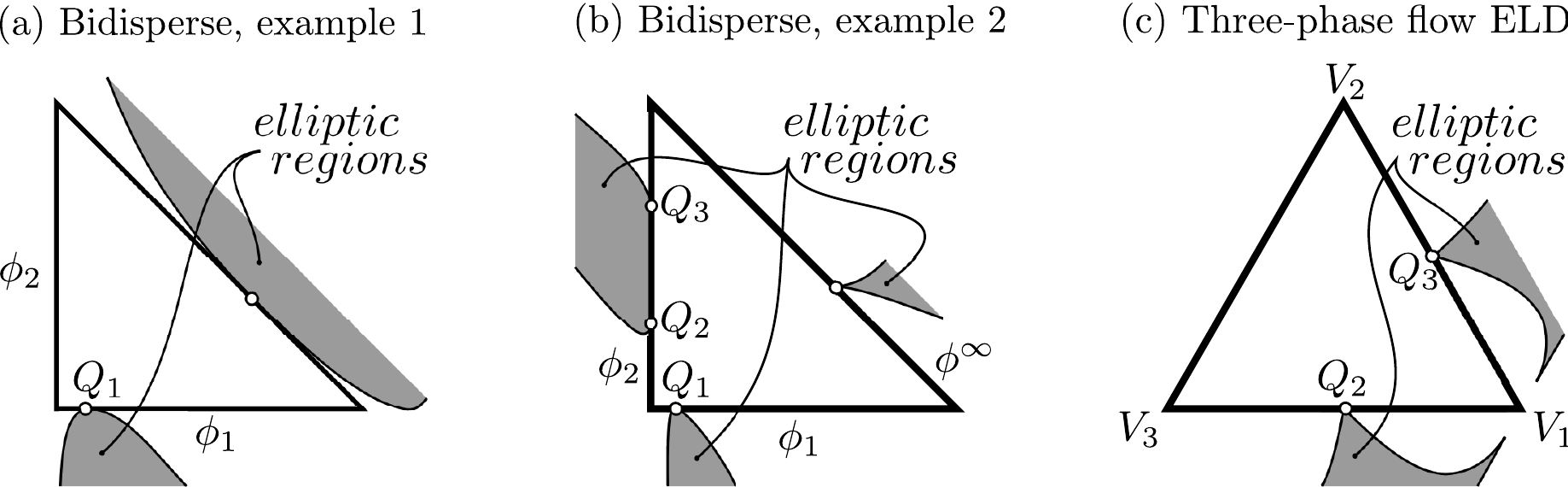}
\caption{Extended domains. % ; see Remark~\ref{rem:qumbilic}. ((el texto refiere a la figura, no al revŽs))
Concentration and saturation triangles are marked by a bold solid line, 
shaded regions contain states with complex eigenvalues of the flux Jacobian. In (a) the Example 1 reveals two quasi-umbilic points at the boundary of elliptic regions; the point $Q_1$ was detected by our analysis. In (b) the Example 2 shows the three quasi-umbilic points detected by the previous analysis, another one appears at the boundary of the elliptic regions. The plot in (c) represents the ELD where the two quasi-umbilic points also belong to the boundary of elliptic regions.}
\label{fig:QuaseUmbilic}
\end{figure}

% Mejor no poner como comentario sino como texto
% \begin{remark}\label{rem:qumbilic}
% As a curiosity we have extended 
For illustration of the behavior of the coincidence points
the flux functions \eqref{fluxfunction} are continuously extended beyond the phase space
% via the hindered-settling velocities \eqref{videf} 
by considering relative velocities \eqref{uidef} without the cut-off \eqref{videf}.
% in order to be smooth outside the domain $\phi \in [0,\phimax]$. 
The % numerical analysis 
visualization
in Fig.~\ref{fig:QuaseUmbilic} shows that the identified quasi-umbilic point happens to belong 
to an elliptic/hyperbolic boundary; thus, it is clear that $\lambda_1(Q_1) = \lambda_2(Q_1)$ holds.
Such a concept is introduced in \cite{panters}, and here, their flux functions 
are also extended by letting the dominated quadratic terms to act outside the physical domain. In this case the quasi-umbilic points also belong to an elliptic/hyperbolic boundary, see Fig.~\ref{fig:QuaseUmbilic}~(c).

For the bidisperse model with parameter specifications (S1)--(S3) two general examples are considered. 
Besides {\bf Example 1} with parameter setting \eqref{parametersetting1}, % as $v_{1\infty} = 1$, $v_{2\infty} = 1/2$, $n_1 = 4$ and $n_2 = 3$,
{\bf Example 2} has the parameters
\begin{align} \label{parametersetting2}
	v_{1\infty} = 1, v_{2\infty} = 1/2, \qquad n_1 = 4.6, n_2 = 1.5 . 
\end{align}
For the {\bf ELD} model (Equal-Lighter-Density fluids) case in Rodr\'{\i}guez-Berm\'udez and Marchesin (2013) the parameters are chosen as $\rho_1 = 2.0 > 1.0 = \rho_2 = \rho_3$.
% \end{remark}

% \section{The wave curve method for the 2$\times$2 system}

\section{Shock classification of Hugoniot locus of origin} \label{sec:shock}

In this Section the types of shocks which are connected to the origin are classified.
The shock classification of Definition~\ref{def:shocks} 
distinguishes three different types of admissible shocks, namely 1-Lax, 2-Lax and overcompressive shocks.
The locations of the shocks on the Hugoniot $\mathcal{H}(O)$
of the origin $O = (0,\,0)^T$ are identified in Lemma~\ref{hlorigin2};
it includes in particular two contact manifolds.
The shock classification builds on the eigenvalue analysis of Section \ref{sec:eig}
by comparing the shock speed and the eigenvalues in each state on the Hugoniot locus.

% \marginpar{reubicar ? TALVEZ entre 5.3 y 5.4}
%
%

A key feature for the determination of the shock type % on the Hugoniot locus of the origin
consists of the  order switch of the relative velocities
at the threshold concentration $\phis$, 
\begin{align} \label{u12ineq}
	\begin{cases}
		u_1(\Phi) > u_2(\Phi) \quad \text{for} \quad \phi \in [0, \phis), \\
		u_1(\Phi) = u_2(\Phi) \quad \text{for} \quad \phi \in \{\phis, \phimax\} , \\
		u_1(\Phi) < u_2(\Phi) \quad \text{for} \quad \phi \in (\phis, \phimax) ,
	\end{cases}
\end{align}
which is a direct consequence of the definition of the relative velocity \eqref{relvelcon}, \eqref{videf},
together with the specifications (S1) and (S2).

A visual guide for the shock classification 
of shocks between the origin and states on the edges $\partial^1$ and $\partial^2$
is given by Fig.~\ref{fig:edgeanalysis},
where the shock speeds are compared to the characteristic speeds. 
The shock classification is essentially obtained by speed comparisons.
%
% The shock characterization according to Definition~\ref{def:shocks} 
% applied to states on both edges can be read off from . 
%
%
% \marginpar{relocate inequality $u_1(\Phi) > u_2(\Phi)$}
%
In the following Theorem, the shock characterization
% of any shock connecting the origin to a state on its Hugoniot locus,
is established for general parameter choices.
%
% The shock classification along the edges $\partial^i, \; i=1,2$
% $\Phi = (0,\,\phi_2)^T \in \partial^2$, 
% can be oriented by Fig.~\ref{fig:edgeanalysis}.
%
%
%\TD{Figura: ...denominar ejes, incluyendo $\phisigma, \phiu$ -   cambiar notacion a $\phi^\sigma, \phi^u$ ? Adapt $\lambda_a, \lambda_b$ to notation of \eqref{lam12:phiab} }

% \marginpar{En figura shocks\_a.pdf, cambiar $\lambda_a, \lambda_b$.}
% \marginpar{En figura shocks\_b.pdf, cambiar los colores.}

\begin{figure}[ht]
\begin{center}
\begin{tabular}{cc}
\includegraphics[width=0.49\textwidth]{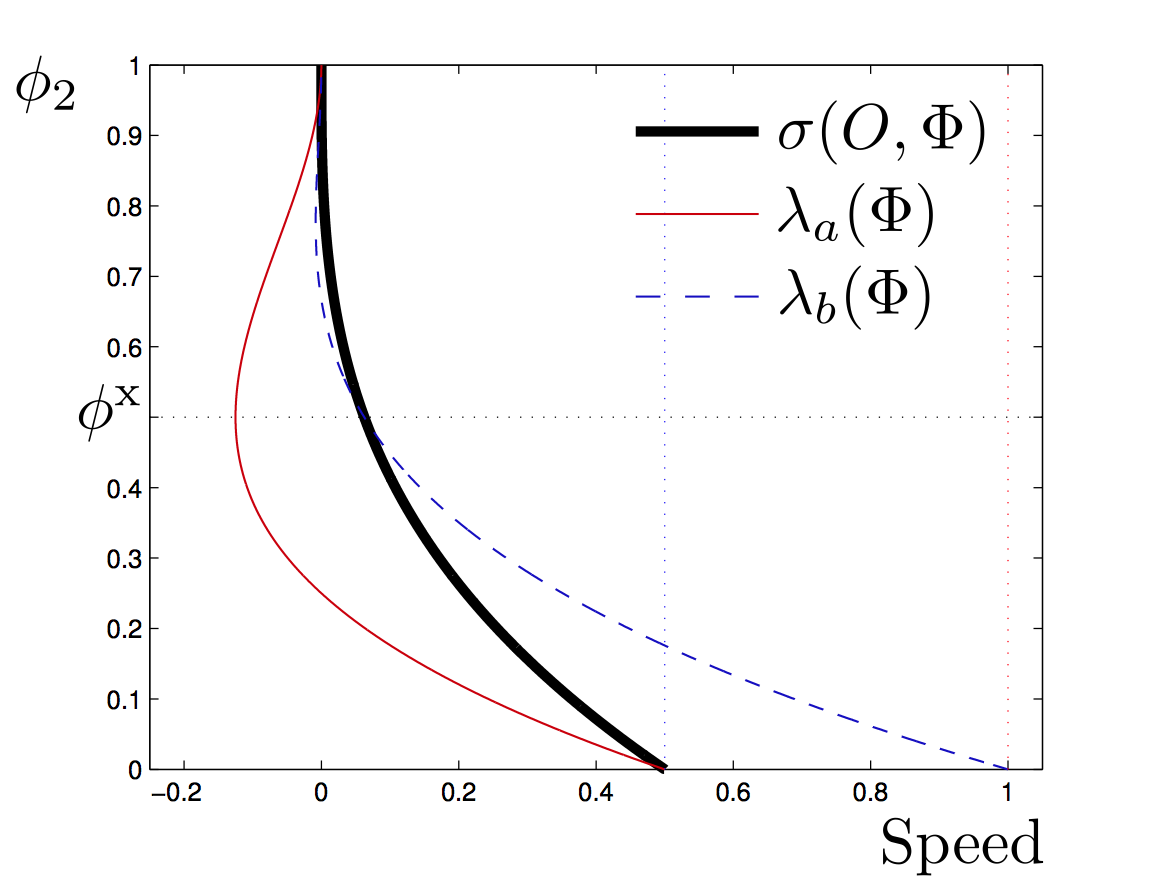} &
\includegraphics[width=0.49\textwidth]{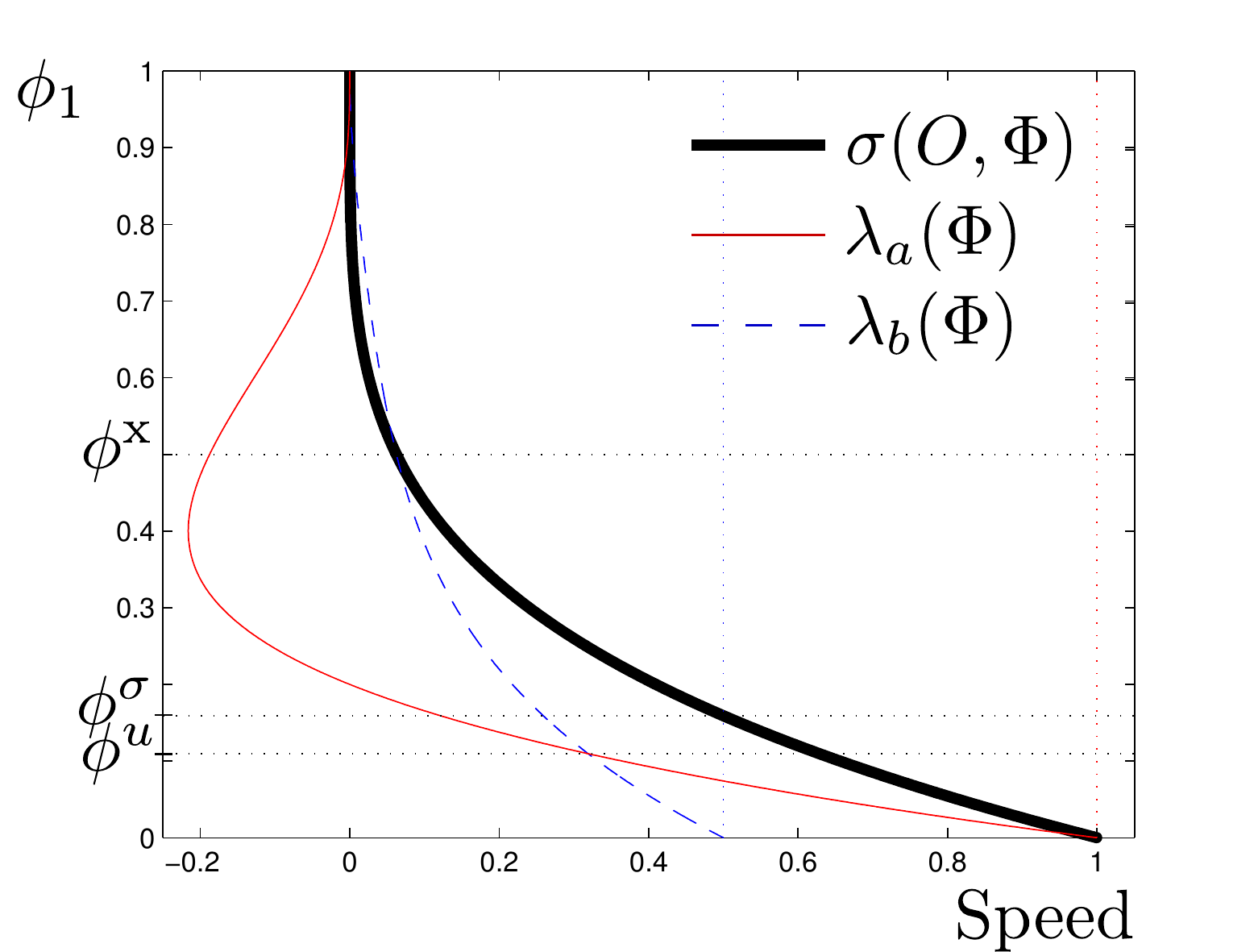} \\
(a) & (b)
\end{tabular}
\end{center}
\caption{Eigenvalues and shock speeds of shocks from the origin to (a) the states $\Phi= (0,\,\phi) \in \partial^2,\, \phi\in[0,\,1]$, 
(b) the states $\Phi= (\phi,\,0) \in \partial^1,\,
\phi\in[0,\,1]$. Bold solid curves represent the shock speed $\sigma(O,\,\Phi)$.
The eigenvalues (characteristic speeds) given in \eqref{lam12:phiab} are represented 
as light solid curves for $\lambda_a(\Phi)$ and dashed curves for $\lambda_b(\Phi)$. 
The dotted horizontal lines represent heights corresponding to $\phis$, $\phisigma$, % = 1 - \,\,^{n_1}\!\!\!\!\sqrt{v_{\infty 2}/v_{\infty 1}}$ 
and $\phiu$, ordered from top to bottom. 
} % caption
\label{fig:edgeanalysis}
\end{figure}

\begin{theorem} \label{teo:classification}
% On the edges 
% \begin{align}
% 	\partial^1 := \{ \Phi: (\phi,0), \phi \in [0,1] \},
% 	\quad
% 	\partial^2 := \{ \Phi: (0,\phi), \phi \in [0,1] \},	
% \end{align}
% On the $\phi_i$-axis 
The Riemann problem % RP(O, $\mathcal{H}(O)$)
\begin{equation*}
\RP(O, \Phi), \; \Phi \in \mathcal{H}(O) = \partial^1 \cup \partial^2 \cup \mathcal{C}(\Phis) \cup \mathcal{C}(\Phimax) % \partial_i, \; i=1,2$ 
\end{equation*}
connecting the origin to a state of its Hugoniot locus % of the origin, 
is solved by a single shock of speed $\sigma = \sigma(O,\Phi)$,
which can be classified as follows according to Definition \ref{def:shocks}.

% \marginpar{separar casos $<$ y $\le$}

\noindent \underline{Classification for $\Phi = (\phi_1, 0)^{\mathrm{T}} \in \partial^1$:}
\begin{equation} \label{classification1}
\begin{array}{rcl}
% & \lambda_{1,\,2}(\Phi) \,\leq\, \sigma \, =\, \lambda_2(O) \,\mbox{ and }\,\lambda_1(O) \,<\, \sigma & \,\mbox{ for }\,\phi_1 = 0 \\
%
\mbox{2-Lax: } & \lambda_{1,\,2}(\Phi) \,\leq\, \sigma \, <\, \lambda_2(O)\,\mbox{ and }\,\lambda_1(O) \,<\, \sigma & \,\mbox{ for }\,\phi_1 \in (0,\,\phisigma), \\
& \lambda_{1,\,2}(\Phi) \,\leq\, \sigma \, < \, \lambda_2(O)\,\mbox{ and }\,\lambda_1(O) \, = \, \sigma & \,\mbox{ for }\,\phi_1 = \phisigma. \\
\mbox{OC: }    & \lambda_{1,\,2}(\Phi) \,<\, \sigma \,<\, \lambda_{1,\,2}(O) & \,\mbox{ for }\,\phi_1 \in (\phisigma,\,\phis). \\
\mbox{1-Lax: } & \lambda_1(\Phi) \, < \, \sigma \,<\, \lambda_{1,\,2}(O)\,\mbox{ and }\,\sigma \, = \, \lambda_2(\Phi) & \,\mbox{ for }\,\phi_1 \in \{ \phis, \phimax \}, \\
& \lambda_1(\Phi) \, < \, \sigma \,<\, \lambda_{1,\,2}(O)\,\mbox{ and }\,\sigma \, < \, \lambda_2(\Phi) & \,\mbox{ for }\,\phi_1 \in (\phis,\,\phimax). \\
%
% & \lambda_1(\Phi) \, < \, \sigma \,<\, \lambda_{1,\,2}(O)\,\mbox{ and }\,\sigma \, = \, \lambda_2(\Phi) & \,\mbox{ for }\,\phi_1 = \phimax 
\end{array}
\end{equation}
The threshold value $\phisigma$ is given as
\begin{align} \label{phisigma}
	\phisigma = 1 - \,\,^{n_1}\!\!\!\!\sqrt{v_{\infty 2}/v_{\infty 1}}.
\end{align}

\noindent \underline{Classification for $\Phi = (0, \phi_2 )^{\mathrm{T}} \in \partial^2$:}
\begin{equation} \label{classification2}
\begin{array}{rcl}
		% &  \lambda_1(\Phi) \, = \, \sigma \,\leq\, \lambda_{1,\,2}(O)\,\mbox{ and }\,\sigma \, < \, \lambda_2(\Phi) & \,\mbox{ for }\,\phi_2 = 0 \\
\mbox{1-Lax: } & \lambda_1(\Phi) \,<\, \sigma \, < \, \lambda_{1,\,2}(O)\,\mbox{ and }\,\sigma \, < \, \lambda_2(\Phi) & \,\mbox{ for }\,\phi_2 \in (0, \phis),  \\
		& \lambda_1(\Phi) \, < \, \sigma \,< \, \lambda_{1,\,2}(O)\,\mbox{ and }\,\sigma \,= \, \lambda_2(\Phi) & \,\mbox{ for }\,\phi_2 \in \{ \phis, \phimax \}. \\
\mbox{OC: }    & \lambda_{1,\,2}(\Phi) \,<\, \sigma \,<\, \lambda_{1,\,2}(O) & \,\mbox{ for }\,\phi_2\in(\phis,\,\phimax).  \\
% \mbox{1-Lax: } & \lambda_1(\Phi) \,<\, \sigma \, < \, \lambda_{1,\,2}(O)\,\mbox{ and }\,\sigma \,=\, \lambda_2(\Phi) & \,\mbox{ for }\,\phi_2 = \phimax
\end{array}
\end{equation}

\noindent \underline{Classification for $\Phi \in \mathcal{C}(\Phis)$:}
\begin{equation} \label{classification3}
\begin{array}{rcl}
	\mbox{1-Lax: } & \lambda_1(\Phi) \,<\, \sigma \, < \, \lambda_{1,\,2}(O)\,\mbox{ and }\,\sigma \,=\, \lambda_2(\Phi). % & \,\mbox{ for }\,\Phi \in \mathcal{C}(\Phis)
\end{array}
\end{equation}

\noindent \underline{Classification for $\Phi \in \mathcal{C}(\Phimax)$:}
\begin{equation} \label{classification4}
\begin{array}{rcl}
	\mbox{1-Lax: } & \lambda_1(\Phi) \,=\, \sigma \, < \, \lambda_{1,\,2}(O)\,\mbox{ and }\,\sigma \,=\, \lambda_2(\Phi). % & \,\mbox{ for }\,\Phi \in \mathcal{C}(\Phimax)
\end{array}
\end{equation}
\end{theorem}	

Strict inequalities in the shock type classification \eqref{classification1}-\eqref{classification4} mean that 
the corresponding shocks are not characteristic shocks. For instance, for $\Phi \in \partial^1$ there are no 1-Lax shocks characteristic at the left datum $O$.
% , for $\Phi \in \partial^2$ there are no 1-Lax shock characteristic at the right datum $\Phi$. 
%
% \marginpar{for $\Phi \in \partial^2$, shock is right characteristic on $\Phimax$?}
% 
% 
% The contact manifold $\mathcal{C}(\Phis)$ is 1-Lax since
% \begin{align}
% \lambda_1(\Phi) < \lambda_2(\Phi) = \sigma(O,\Phi) < \lambda_1(O) < \lambda_2(O) \quad \text{for all} \quad \Phi \in \mathcal{C}(\Phic)
% \end{align}
% where the equality $ \lambda_2(\Phi) = \sigma(O,\Phi)$
% lets the shock to be characteristic with respect to the second family.
% 
% a. Figure~\ref{fig:edgeanalysis}.b. 
% The shock classification along the edge $\phi_2 \equiv 0$ with $\Phi = (\phi_1,\,0)^T \in \partial^1$, can be oriented by
% The eigenvalues $\lambda_a$ are visualized as the dashed curve in Figure~\ref{fig:edgeanalysis}.

\begin{proof}
The proof is done in two main steps: \\
$\quad$ 1.- The first step is to relate the characteristic speeds $\lambda_{1,2}(O)$ and $\lambda_{a,b}(\Phi)$ to the shock speed $\sigma(O,Ê\Phi)$ in dependence of the position of $\Phi$. \\
$\quad$ 2.- The second step consists in determining which shock  classification applies according to Definition \ref{def:shocks}. \\

Table \ref{tab:shockclassification} gives an overview on the speed magnitudes for right states on the axes $\partial^1$ and $\partial^2$. \\[2mm]

\begin{table}[ht]
\centering
\begin{tabular}{rl | lr | lr | c} %   lllll}
% \hline  
  % & & \multicolumn{3}{|c||}{$<Ê\sigma(O, \Phi)$} & \multicolumn{2}{c||}{$>Ê\sigma(O, \Phi)$} & \\ \hline \hline
  & & \multicolumn{2}{|c|}{$<\sigma(O, \Phi)$} & \multicolumn{2}{c|}{$>\sigma(O, \Phi)$} & \\ \hline \hline
%  &\cline{1-3} &  \cline{4-5} \\ \hline
$\partial^1:$ & $\phi < \phi^\sigma$ & $\lambda_{a,b}(\Phi)$ & $\lambda_1(O)$ & & $\lambda_2(O)$ & 2-Lax \\
		    & $\phi \in  (\phi^\sigma, \phis$) & $\lambda_{a,b}(\Phi)$ & & & $\lambda_{1,2}(O)$ & OC \\
		    & $\phi > \phis$ & $\lambda_a(\Phi)$ & & $\lambda_b(\Phi)$ & $\lambda_{1,2}(O)$ & 1-Lax \\ \hline
$\partial^2:$ & $\phi < \phis$ & $\lambda_a(\Phi)$ & & $\lambda_b(\Phi)$ & $\lambda_{1,2}(O)$ & 1-Lax \\
		   & $\phi > \phis$ & $\lambda_{a,b}(\Phi)$ & &  & $\lambda_{1,2}(O)$ & OC % \\ \hline \hline
\end{tabular}
\vspace{1cm}
\caption{Shock classification on Hugoniot locus of origin for states  on the axes $\partial^1$ and $\partial^2$.
First column: Location on the axis. Second column: Characteristic speeds less than $\sigma(O, \Phi)$.
Third column: Characteristic speeds larger than $\sigma(O, \Phi)$.
Forth column: Shock type.
}
\label{tab:shockclassification}
\end{table}

%%%%%%%%%%%%%%%%%%%%%%%%%%%%%%%%%%%%%%%%%%%%%%%%%%
		 
% for all $\Phi \in \partial^i \backslash O$. 
\noindent % $\bullet$ 
\underline{$\lambda_2(O)$ compared to $\sigma(O,\,\Phi)$ on $\partial^1$ and $\partial^2$}

On the axes $\Phi \in \partial^i \backslash O$ the shock speed is limited as
\begin{align} \label{sl2}
		\sigma(O,\Phi) = v_i(\Phi) < v_{\infty i} 
		=
	\begin{cases}
		 v_{\infty 1} = \lambda_2(O), & \quad \text{if} \quad i=1, \\
		 v_{\infty 2}  = \lambda_1(O) < \lambda_2(O), & \quad \text{if} \quad i=2,
	\end{cases}
\end{align}
since the shock speed
\begin{align*}
	\sigma(O,\,\Phi) = v_i(\Phi)=(1-\phi_i)u_i(\Phi) = v_{\infty i} (1-\phi_i)^{n_i}, \quad i=1,\,2,
\end{align*}
is monotonically decreasing on both axes % $\partial^i, \; i=1,2$.  
$\partial^1$ and $\partial^2$. 
(Note that \eqref{sl2} excludes 2-Lax shocks on edge $\partial^2$.)
\\[2mm]

%%%%%%%%%%%%%%%%%%%%%%%%%%%%%%%%%%%%%%%%%%%%%%%%%%

%$\bullet$ 
\noindent
\underline{$\lambda_a(\Phi)$ compared to $\sigma(O,\,\Phi)$ on $\partial^1$ and $\partial^2$} % 

The eigenvalue $\lambda_a(\Phi)$ 
as specified in \eqref{lam12:phiab} 
leads to common properties for both axes.
%
% For the eigenvalues along the edge, first notice % from (\ref{lam12:phi1}.b) that  
For all points $\Phi \in \partial^i \backslash O, \; i \in \{ 1,2 \}$, along the axes one has % respectively
\begin{align}
	\begin{split}
	\label{l1s}
	\lambda_a(\Phi) 
	& = [ 1-(1+n_i)\phi_i ] u_i(\Phi) \\
	& < (1-\phi_i) u_i(\Phi) = v_i(\Phi) = \sigma(O,\Phi), \quad i=1,2 ,
	\end{split}
\end{align}
due to  properties that apply on the axes, namely, the speed convertibility \eqref{vedge} and the shock speed \eqref{speedL}.  \\[2mm]
%

% \begin{align} \label{sl2}
%	\sigma(O,\Phi) = v_1(\Phi) < v_{\infty 1} = \lambda_2(O) .
% \end{align}
%
%%%
% whereas 
% \begin{align} \label{sl2}
%	\sigma(O,\,\Phi) = v_2(\Phi) < v_{\infty 2}  = \lambda_1(O) < \lambda_2(O) % \lambda_{1,\,2}(O)
% \end{align}
% \marginpar{$<$ versus $\le$}
%
%
% Secondly, notice that at $\phi_1 = \phic$ there is a change, say  satisfies that 

%%%%%%%%%%%%%%%%%%%%%%%%%%%%%%%%%%%%%%%%%%%%%%%%%%

%%%%%%%%%%%%%%%%%%%%%%%%%%%%%%%%%%%%%%%%%%%%%%%%%%

%$\bullet$ 
\noindent
\underline{$\lambda_b(\Phi)$ compared to $\sigma(O,\,\Phi)$ on $\partial^1$ and $\partial^2$}

The composition of the eigenvalue $\lambda_b(\Phi)$ depends on 
the inequality \eqref{u12ineq}
between the relative velocities $u_1(\Phi)$ and $u_2(\Phi)$,
and induces quite different behaviors on the edges. 
On the axis $\partial^1$ 
the eigenvalue $\lambda_b(\Phi) $ behaves as
\begin{align}
\begin{cases} \label{laswitch}
		\lambda_b(\Phi)  < \sigma(O,\,\Phi) \quad \text{for} \quad \phi_1 \in (0,\,\phis), \\
		\lambda_b(\Phi) = \sigma(O,\,\Phi) \quad \text{for} \quad \phi_1 \in \{ \phis, \phimax \}, \\		
		\lambda_b(\Phi) > \sigma(O,\,\Phi) \quad \text{for} \quad \phi_1 \in  (\phis,\,\phimax) ,
	\end{cases}
\end{align}
because on $\partial^1$, the inequalities \eqref{u12ineq}
between $u_1(\Phi)$ and $u_2(\Phi)$
implies that for the interval $\phi_1 \in [0,\,\phis)$
the eigenvalue (\ref{lam12:phiab}) satisfies
\begin{align*}
	\lambda_b(\Phi) = u_2(\Phi) - \phi_1 u_1(\Phi) < (1 - \phi_1) u_1(\Phi) =  \sigma(O,\,\Phi) ,
\end{align*}
and for the other interval, where $\phi_1 \in  (\phis,\,\phimax]$, the inequality sign switches.

On the axis $\partial^2$
one has
\begin{align} \label{l2switch}
	\begin{cases} 
		\lambda_b(\Phi) > \sigma(O,\Phi) \quad \text{for} \quad \phi_2 \in [0, \phis), \\
		\lambda_b(\Phi) = \sigma(O,\Phi) \quad \text{for} \quad \phi_2 \in \{�\phis, \phimax \}, \\
		\lambda_b(\Phi) < \sigma(O,\Phi) \quad \text{for} \quad \phi_2 \in (\phis, \phimax ) ,
	\end{cases}
\end{align}
since, from the inequalities \eqref{u12ineq} for $\phi_2 \in [0, \phis)$,
one obtains
\begin{align*}
	\lambda_b(\Phi) = u_1(\Phi) - \phi_2 u_2(\Phi) > (1-\phi_2)u_2(\Phi) =  \sigma(O,\,\Phi)
\end{align*}
and for $\phi \in (\phis,\phimax]$ the inequality sign switches. 

On the edge $\partial^2$ the value of $\lambda_b(\Phi)$ % given in (\ref{lam12:phiab}.b)
determines whether the shock is 1-Lax or over-compressive. The shock type is decided by the relative magnitudes
of $u_1(\Phi)$ contra $u_2(\Phi)$: The equality $u_1(\Phi) = u_2(\Phi)$ only holds for $\phi_2 = \phis$ (and $\phi_2 = \phimax$).
On edge $\partial^1$ there is a coincidence of eigenvalues, see also Lemma~\ref{lem:quasiumbilic}. \\[2mm]

% [?? is there a relation of the coincidence point to $\phi^{\sigma}$. On $\partial^2$ can the possibly two coincidences affect the shock characterization?] 

% edge \partial^1

%%%%%%%%%%%%%%%%%%%%%%%%%%%%%%%%%%%%%%%%%%%%%%%%%%

%$\bullet$ 
\noindent
\underline{$\lambda_1(O)$ compared to $\sigma(O,\,\Phi)$ on $\partial^1$ and $\partial^2$}

% \marginpar{Poner al inicio de la prueba, lema con coincidencias antes; lema correcto? Ahora si}

Note that, along both axes $\sigma(O,\Phi)$ is a monotonically decreasing function with $\sigma(O,\Phi^{\infty})=0$. 
By \eqref{sl2} is assured that $\sigma(O,\Phi) < \lambda_1(O)$ holds for all states on the axis $\partial^2 \backslash O$. 
However, this does not hold for all states on the axis $\partial^1 \backslash O$.

Since 
$\lim_{\varepsilonÊ\rightarrow 0} \sigma(O, (\varepsilon, 0)) = \lambda_2(O) > \lambda_1(O) > 0$ on the axis 
$\partial^1$, there exists a state $\Phisigma=(\phisigma,0)^{\mathrm{T}}$ such that
$\sigma(O,\Phisigma) = \lambda_1(O)$.
In this state $\Phisigma$ 
% =(\phisigma,0)^{\mathrm{T}}$,
the equality 
\begin{align*}
	\sigma(O,\Phisigma) = v_1(\Phisigma) = 
	v_{\infty 1} ( 1 - \phi_{\sigma})^{n_1}  = v_{\infty 2}  = \lambda_1(O) 
\end{align*}
% Moreover, since $\sigma(O,\,\Phi) = v_{\infty 2}$ holds only for 
% such that the evaluation 
% of $v_{\infty 2} = v_1(\Phisigma)$
% gives $\phisigma$
holds, from which the characterization of $\phisigma$ by \eqref{phisigma} can be deduced.
The shock speed $\sigma(O,\Phi)$ relates to $\lambda_1(O)$ 
in dependence of $\phisigma$ as
\begin{align} \label{ocineq}
\begin{cases}
		\sigma(O,\,\Phi) > \lambda_1(O) \quad \text{for} \quad \phi_1 \in [0,\,\phisigma), \\
		\sigma(O,\,\Phi) = \lambda_1(O) \quad \text{for} \quad \phi_1 = \phisigma, \\ % \in \{ \phisigma, \phimax \} , \\		
		\sigma(O,\,\Phi) < \lambda_1(O) \quad \text{for} \quad \phi_1 \in (\phisigma,\,\phimax] .
	\end{cases}
\end{align}

Together with the previously established \eqref{sl2}, \eqref{l1s} and \eqref{laswitch} % ,  \eqref{laswitch})
the discrimination \eqref{ocineq} implies that,
on the axis $\partial^1$,
for $\phi < \Phisigma$ 
there is a 2-Lax shock,
whereas 
for $\Phisigma < \phi < \phis$
the shock is over-compressive.

By \eqref{l1s}, \eqref{laswitch}, for $\phi \in (\phis, \phimax)$, on axis $\partial^1$ there is a clear separation between the shock speeds:
\begin{align} \label{shockorder}
	\lambda_1(\Phi) = \lambda_a(\Phi) < \sigma(O,\Phi) < \lambda_b(\Phi) = \lambda_2(\Phi) . 
\end{align}
This separation allows an association of the eigenvalues according to \eqref{lambdaorder}.
The same way,
by \eqref{l1s}, \eqref{l2switch} for $\phi \in (0, \phis)$ on $\partial^2$ there is a clear separation between the shock speeds
\eqref{shockorder} that establishes the eigenvalues order \eqref{lambdaorder}.
%
%\marginpar{Para valores $n_1 \gg n_2$ puede haber coincidencia}
For $\phi \in (\phis, \phimax)$ there is no clear eigenvalue separation, which however does not affect the 
fact that the shock is over-compressive.

Now we are able to conclude the shock classification on the axes:
Putting  the inequalities
\eqref{sl2},
\eqref{l1s}, 
% \eqref{l1switch} 
% respective
\eqref{l2switch}
together gives the shock classification
% \eqref{classification1} 
% respective the classification
 \eqref{classification2} for right states on the axis $\partial^2$.
Putting  the inequalities
\eqref{sl2}, 
\eqref{l1s},
\eqref{laswitch}, \eqref{ocineq}
together gives the shock classification \eqref{classification1}
for right states on the axis $\partial^1$. \\[2mm]
%

% contact manifold

%%%%%%%%%%%%%%%%%%%%%%%%%%%%%%%%%%%%%%%%%%%%%%%%%%

%$\bullet$ 
\noindent
\underline{Contact manifold $\mathcal{C}(\Phis)$} 

The shock classification for states on the contact manifold $\Phi \in \mathcal{C}(\Phis)$ assures that all states are connected to the origin by a 1-Lax shock:
Since the functions $v_1(\Phi)$ and $v_2(\Phi)$ take their maximum in the origin $O$ and the origin is not part of the contact manifold, one has 
the strict inequality
\begin{align} \label{sigmainequality}
	\sigma(O,\Phi) = v_1(\Phis) < v_{\infty 1} = \lambda_1(O) < v_{\infty 2} = \lambda_2(O) .
\end{align}
By the
eigenvalue characterization \eqref{eq:lambda2onC} for states on $\mathcal{C}(\Phis)$, 
and noting that 
$\Delta_{\Phi}>0$ for all  $\Phi \in \mathcal{C}(\Phis)$,
 one gets
\begin{align} \label{sigmaphis} 
 \lambda_1(\Phi) < \sigma(O,\Phi) = v_{1}(\Phi) = \lambda_2(\Phi) 
\end{align}
 for all  $\Phi \in \mathcal{C}(\Phis)$.
The properties \eqref{sigmainequality} and \eqref{sigmaphis} 
are summarized in the shock classification \eqref{classification3}. \\[2mm]

%%%%%%%%%%%%%%%%%%%%%%%%%%%%%%%%%%%%%%%%%%%%%%%%%%

%$\bullet$ 
\noindent
\underline{Contact manifold $\partial^{\infty} = \mathcal{C}(\Phi^{\infty})$}

For states on the maximum packing manifold $\Phi \in \partial^{\infty}$
one has vanishing eigenvalues,
\begin{align*}
	0 = \lambda_1(\Phi) = \lambda_2(\Phi) =  \sigma(O,\Phi) < \lambda_1(O) 
	< \lambda_2(O) ,
\end{align*}
which leads to classification \eqref{classification4}.
\end{proof}

%When the eigenvalues in \eqref{lam12:phiab} do not coincide we identify $\lambda_a(\Phi)$ with the first eigenvalue $\lambda_1(\Phi)$.

% \subsection{1-Lax locus}
\begin{figure}[htb]
\centering
\includegraphics[width = 0.4 \textwidth]{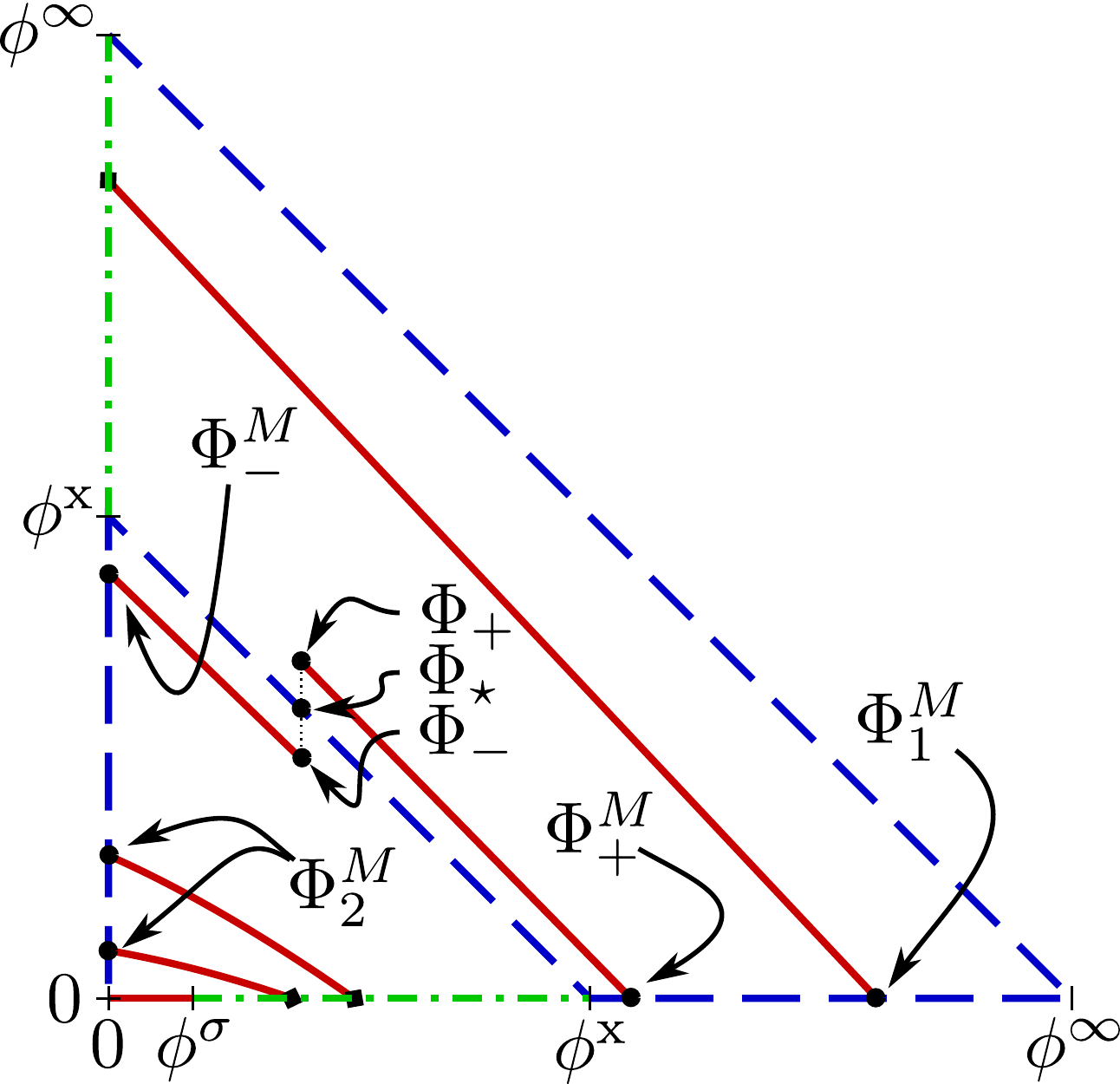}
\caption{Different shock types for states connected to the origin. Continuous curves are 2-Lax shocks, dashed curves are 1-Lax shocks and dotted-dashed are over-compressive shocks. The Hugoniot locus of the origin comprises the three edges and the contact manifolds $\mathcal{C}(\Phis)$, $\mathcal{C}(\Phimax)$.
Also some 2-Lax shock curves from arbitrarily chosen generic points $\Phi_1^M$, $\Phi_2^M$, $\Phi_-^M$ and $\Phi_+^M$ are indicated. 
Reference values on the horizontal axis $\partial^1$ are $\phisigma$, $\phis$ and $\phimax$, on the vertical axis $\partial^2$ are $\phis$ and $\phimax$.}
\label{fig:shClassification}
\end{figure}

Figure~\ref{fig:shClassification} displays the different shock types for states connected with the origin, 
as elaborated in Theorem \ref{teo:classification}.
The 1-Lax states of $\mathcal{H}(O)$ are shown  as dashed lines.
%\TD{Buscar lugar para esto: Even if there may be family coincidence at $\Phi = (\phi_c,\,0)^T$ such that $\phi_c$ is a zero root of $R(\phi) := (1 - n_1\phi)(1 - \phi)^{n_1 - n_2} - v_{\infty 2}/v_{\infty 1}$, because the eigenvalues in \eqref{lam12:phi1}...}

% \marginpar{Integrate in theorem?}
%
% \begin{corollary}
States on the contact manifold $\mathcal{C}(\Phis)$ satisfy the 1-Lax conditions being characteristic in the second family at the right state. 
In the same manner, the states on the maximum packing $\mathcal{C}(\Phimax)=\partial^{\infty}$  
satisfy to be 1-Lax shocks which are characteristic at the right state with respect to both families. 
On the edge $\partial^2$, at $\Phis = (0,\,\phis)^{\mathrm{T}}$ and % \in \partial^1$ and 
$\Phimax = (0,\,\phimax)^{\mathrm{T}}$ % = \partial^1 \cup \partial^{\infty}$
the shocks are 1-Lax and characteristic at the right states, 
say $\sigma(O,\,\Phis) = \lambda_2(\Phis)$ and $\sigma(O,\,\Phimax) = \lambda_{2}(\Phimax)$.
On the edge $\partial^1$, % $\phi_2 \equiv 0$, 
at $\Phis = (\phis,\,0)^{\mathrm{T}}$, $\Phisigma = (\phisigma,\,0)^{\mathrm{T}}$ and $\Phimax = (\phimax,\,0)^{\mathrm{T}}$ the shocks are 2-Lax, over-compressive and 1-Lax, respectively, being characteristic at left or right states, say $\sigma(O,\,\Phis) = \lambda_1(O)$, $\sigma(O,\,\Phisigma) = \lambda_2(\Phisigma)$ and $\sigma(O,\,\Phimax) = \lambda_{2}(\Phimax)$.
%
%\marginpar{double}
Several states located in Fig.~\ref{fig:edgeanalysis}~(b) have characteristic shocks:
If $\Phi \in \mathcal{C}(\Phis)$ then the shock is right characteristic for the second family, 
if $\phi = \phimax$ then it is characteristic for both families, 
and if $\Phi = (0,\,\phisigma)^{\mathrm{T}}$ then the shock is 2-Lax but left characteristic in the first family.

% \end{corollary}
%\begin{proof}
%Notice that the shock speed is $\sigma = v_1(\Phic) = v_2(\Phic)$ for all $\Phic$ belonging to $\mathcal{C}(\Phis)$ or $\mathcal{C}(\Phimax)$. From Eq.~\eqref{eigen:cstar} and inequality \eqref{l2switch} 
%on has
%\begin{align}
%	% \lambda_{1}(\Phis) = 
%	\lambda_{2}(\Phis) = \sigma(O,\,\Phis),
%	\quad
%	\sigma(O,\,\Phimax) = \lambda_{1}(\Phimax) = \lambda_{2}(\Phimax)=0
%\end{align}
%such that the corollary follows.
%\end{proof}

%

% \marginpar{mejorar / omitir unos comentarios?}

% OMITIR POR SER MENOS RELEVANTE?

% and at the origin it takes the value $\sigma(O,O)=v_i(O)=v_{\infty i}$.
%
% $\bullet$ observing that the eigenvalues in the origin are different, $\lambda_1(O)=v_{\infty 2} < v_{\infty 1} = \lambda_2(O)$,
% one has to distinguish the situation on the edges.
%

%%%%%%%%%%%%% REMARK

% ((ÀÀA donde mover??))
% \begin{remark}
% The order 
% $\lambda_a(\Phi) < \sigma(O,\Phi)$ means for the shock classification that 
% the eigenvalue $\lambda_a$ can be 
% either associated with $\lambda_1$ or subsumed by the notation $\lambda_{1,2}$, 
% which contains both $\lambda_1$ and $\lambda_2$ without further specification.
%
% Complementary, the eigenvalue $\lambda_b$
% can be identified as $\lambda_2$ if $\lambda_b \ge \sigma(O,\Phi)$
% or subsumed by  $\lambda_{1,2}$ if $\lambda_b < \sigma(O,\Phi)$.
%% \end{remark}

\section{Riemann problems} \label{sec:rp}

In this Section, we construct the solution of the Riemann problems $\RP(O,\,\Phi)$ and $\RP(\Phi, \, \Phimax)$,
where $O = (0,\,0)^{\mathrm{T}}$, $\Phimax \in \partial^{\infty}$, and $\Phi \in \mathcal{D}_{\Phimax}$ 
is a generic right or left state in the phase space. These Riemann problems are derived from the standard initial condition \eqref{inithom}. % of the batch settling setup.

% \marginpar{reubicar?}
% ((menos relevante))
% The determination of states where shock and characteristic speeds coincide is decisive in the construction of Riemann solutions. 
% For example, a wave curve comprising a rarefaction followed by a shock may exist when such a shock is characteristic at its left state, 
% {\it i.e.}, with a speed that matches the characteristic speed at the end point of the precedent rarefaction. 
%Another important case occurs at the transition between a 1-Lax shock and an over-compressive shock, in such a case the characteristic speed $\lambda_2(\Phi^+)$ matches the shock speed $\sigma(\Phi^-,\Phi^+)$.
% \marginpar{$\sigma=\sigma(\cdot, \cdot)$ ?}
%
%

\subsection{The Riemann problem $\RP(O,\,\Phi)$}

% \subsection*{Construction of $\RP(O,\Phi)$}
%
The behavior of solutions to the Riemann problem $\RP(O,\Phi)$
depends on the position of $\Phi$ in the interior of $\mathcal{D}_{\Phimax}$
with respect to the contact manifold $\mathcal{C}(\Phis)$,
which splits 
% For the Riemann problem $\RP(O,\,\Phi)$ let us split the 
the domain into two regions:
\begin{eqnarray*}
\label{dom:D-}
\mathcal{D}_{\mathrm{x}}^- &:=& \{ \Phi \in \mathcal{D}_{\Phimax} \,:\, \phi < \phis \}, \\
\label{dom:D+}
\mathcal{D}_{\mathrm{x}}^+ &:=& \{ \Phi \in \mathcal{D}_{\Phimax} \,:\, \phi > \phis \},
\end{eqnarray*}
where
%
% \marginpar{Cambie notacion. Espero haber cambiado todos}%$\mathcal{D}_*^+$, $\mathcal{D}_{\mathrm{x}}^+$, $\mathcal{D}_{\max}^+$?}
%
$\mathcal{C}(\Phis) = \mathrm{cl} \bigl( \mathcal{D}_{\mathrm{x}}^- \bigr) \cap \mathrm{cl} \bigl( \mathcal{D}_{\mathrm{x}}^+ \bigr)$
marks the intersection of the closures of those two regions.
% A decisive feature of both domains is that 
The main difference between both domains is the opposite direction of the characteristic speeds on the integral curves.
In $\mathcal{D}_{\mathrm{x}}^-$ the second eigenvalue increases from the edge $\partial^1$ % $\phi_2 \equiv 0$ 
to the edge $\partial^2$, % $\phi_1 \equiv 0$ 
and, reversely, in $\mathcal{D}_{\mathrm{x}}^+$ the second eigenvalue increases from the edge $\partial^2$ % $\phi_1 \equiv 0$ 
to the edge $\partial^1$. % $\phi_2 \equiv 0$.
% \marginpar{Se puede comprobarlo?}
%
See the orientation of the second rarefactions in Fig.~\ref{fig:curves2x2},
where the arrows point into the directions of increasing eigenvalues.
%
% By the direction of the integral curves,
%

A Riemann solution from $O$ to any state $\Phi$ in $\mathcal{D}_{\Phimax}$ generally consists of a 1-Lax shock
followed by a 2-Lax shock.
For $\Phi$ belonging to $\mathcal{D}_{\mathrm{x}}^-$ 
the middle state that intersects the 1-wave with the 2-wave is denoted by $\Phi_-^M$, and for $\Phi$ belonging to $\mathcal{D}_{\mathrm{x}}^+$ the middle state is denoted by $\Phi_+^M$.
In both cases the middle state $\Phi^M$ belongs to the 1-Lax locus of the origin $O$.
% First, we denote by 
Since both waves are shocks any
middle state $\Phi_+^M$ or $\Phi_-^M$ is located at $\mathcal{H}(\Phi) \cup \mathcal{H}(O)$.
% the Hugoniot locus of $\Phi$ 
It turns out that $\Phi_+^M$ belongs to the edge $\partial^1$ % $\phi_2 \equiv 0$ 
and $\Phi_-^M$ belongs to the edge $\partial^2$. % $\phi_1 \equiv 0$. 

%
%
% The Riemann solution consists of the 1-Lax shock from $O$ to $\Phi^M$ and the 2-Lax shock from $\Phi^M$ to $\Phi$. Therefore, the solution of the Riemann problem $RP(O,\,\Phi)$ for $\Phi$ as a right state in the interior of $\mathcal{D}_{\Phimax} \, \backslash \, \mathcal{C}(\Phis)$ is known.

For a right state $\Phi \in \mathcal{D}_{\mathrm{x}}^-$ the solution of $RP(O,\,\Phi)$ 
comprises the 1-Lax shock from $O$ to a state $\Phi^M = (0,\,\phi^M)^T$ such that $\phi^M \in (0,\,\phis)$
and the 2-Lax shock from $\Phi^M$ to $\Phi$.
%^- \in \mathcal{H}(\Phi_-^M) \cap \mathcal{D}_{\mathrm{x}}^-$.
%
Similarly, for a right state $\Phi \in \mathcal{D}_{\mathrm{x}}^+$ the solution of $RP(O,\,\Phi)$ 
comprises the 1-Lax shock from $O$ to a state $\Phi^M = (\phi^M,\,0)^T$ such that $\phi^M \in (\phis,\,\phimax)$
and the 2-Lax shock from $\Phi^M$ to $\Phi$.
%^+ \in \mathcal{H}(\Phi_+^M) \cap \mathcal{D}_{\mathrm{x}}^+$,
Such cases are depicted in Fig.~\ref{fig:shClassification}; the former as $\Phi_- \in \mathcal{D}_{\mathrm{x}}^-$ and the latter as $\Phi_+ \in \mathcal{D}_{\mathrm{x}}^+$.

% \marginpar{Mover al lema \ref{lem:hl1} y extenderlo?}
The solution of the Riemann problem $RP(O,\,\Phi)$ with $\Phi$ on the Hugoniot locus of $O$ 
comprises a single shock from $O$ to $\Phi$ with a classification that depends on the position of $\Phi$ and is elaborated in Theorem \ref{teo:classification}:
\begin{enumerate}
\item For $\Phi = (\phi,\,0)^{\mathrm{T}}$ with $\phi \in (\phisigma,\,\phis)$ the shock is over-compressive,
\item For $\Phi = (0, \, \phi)^{\mathrm{T}}$ with $\phi \in (\phis,\,\phimax)$ the shock is over-compressive,
\item For $\Phi = (\phi,\,0)^{\mathrm{T}}$ with $\phi \in (0,\,\phisigma)$ the shock is 2-Lax, 
\item For any other $\Phi$ in $\mathcal{H}(O)$ the shock is 1-Lax.
%
% \item For $\Phi$ belonging to the 1-Lax locus the shock is 1-Lax,
% \item For $\Phi = (0,\,\phi_2)^{\mathrm{T}}$ with $\phi_2 \in (0,\,\phisigma)$ the shock is 2-Lax, 
% \item For any other $\Phi$ in $\mathcal{H}(O)$ the shock is over-compressive.
\end{enumerate}
%

% TRIPLE SHOCK RULE
For a state $\Phi \in \mathcal{C}(\Phis)$,  
% is a state between $\Phi^-$ and $\Phi^+$ belonging to 
the solution of the Riemann problem 
$RP(O,\,\Phi)$ consists of a single 1-Lax shock from $O$ to $\Phi$.
The two consecutive shocks from $O$ to $(0,\,\phis)^T$ and from $(0,\,\phis)^{\mathrm{T}}$ to $\Phi$ have both the same speed $v_1(\phis) = v_2(\phis)$,
which in turn is the same speed of  the direct shock from $O$ to $\Phi$.
Since all these shocks have the same speed, the middle state $(0,\,\phis)^{\mathrm{T}}$ is ``invisible'' in the solution profile in the physical space;
%this is the essence of the Tripe-Shock Rule, see for instance \cite{HYPcfm, cfm}. 
this is because of Lemma~\ref{TripleShock}.
% lo mismo otra vez:
% and the shock speeds are all equal, thus the solution comprising the two consecutive shocks from $O$ to $(\phis,\,0)^{\mathrm{T}}$ and from $(\phis,\,0)^{\mathrm{T}}$ to $\Phis$ is exactly the same as the solution comprising the single shock from $O$ to $\Phis$.
%
% {\color{red}
Therefore, the solution structure for $\Phi \in \mathcal{D}_{\Phi^\infty}$ is represented as
$$O \xrightarrow{\mathrm{1-Lax}} \Phi^M \xrightarrow{\mathrm{2-Lax}} \Phi.$$
For $\Phi \in \mathcal{C}(\Phis)$ 
the middle state $\Phi^M$ may assume any other state on the same contact manifold and even collapse with $\Phi$.

% CONTINUOUS DEPENDENCE
Please refer to Fig.~\ref{fig:shClassification} and notice that as any $\Phi_- \in \mathcal{D}_{\mathrm{x}}^-$
tends to a $\Phi \in \mathcal{C}(\Phis)$ % $\Phi_\star$, 
the middle state $\Phi_-^M$ tends to $(0,\,\phis)^T$.
Similarly, as $\Phi_+ \in \mathcal{D}_{\mathrm{x}}^+$
tends to a $\Phi  \in \mathcal{C}(\Phis)$, the state $\Phi_+^M$ tends to $(\phis,\,0)^{\mathrm{T}}$.
Thus, we notice a continuous dependence of solutions to the Riemann problem $RP(O,\,\Phi)$ on the right datum.

\subsection{The Riemann problem $RP(\Phi,\,\Phimax)$}

% \marginpar{parameter specification: put parts before}
%
% The numerical approach of the wave curve method reveals the dependency of the solution structure
% for the sedimentation in the benchmark depends 
% on a specific parameter election.

In this Section, the dependence of the solution structure on the exponents $n_1$ and $n_2$ 
is illustrated by two examples, Example 1 and Example 2 
with the corresponding parameter setting \eqref{parametersetting1} and \eqref{parametersetting2}, respectively.
%
% for the slow down process. ???
%
% In Example 1, the parameters are chosen as in \eqref{parametersetting1},
% whereas in Example 2, 
%
% the exponents are set as $n_1 = 4.6$ and $n_2 = 1.5$, keeping the remaining parameters, see Remark~\ref{rem:qumbilic}.
%
% In both examples we used
% $$v_{\infty 1} = 1,\quad v_{\infty 2} = 1/2,\quad\mbox{and}\quad \phimax = 1.$$
%
%The two cases considered are:
%\begin{enumerate}
 %\item $n_1 = 4.0,\quad n_2 = 3.0$. (As the Figures in this work.)
 %\item $n_1 = 4.6,\quad n_2 = 1.5$.
% \end{enumerate}
%
% \marginpar{Figure} 
%
In the Example 1, the solution consists of a simple 1-wave comprising a shock followed by a rarefaction or a single rarefaction. 
The solution structure of Example 2 depends on the existence of one or two detached inflection curves for the first characteristic family.
As $\mathcal{C}(\Phimax)$ is a contact manifold, the goal is to consider a generic state $\Phi \in \mathcal{D}_{\Phimax}$ 
from which waves reaching any state in $\Phi^\infty \in \partial^\infty$ are constructed.
% (All states $\Phi^+$ belong to a small subset of $\partial^\infty$.)

\ \\
\noindent
% {\color{red}
{\it Construction for Example 1}.
The construction for Example 1 can be orientated by the integral curves and the inflection manifold in Figure~\ref{fig:curves2x2}.
In Figure~\ref{fig:curves2x2}~(a) the inflection manifold $\mathcal{I}_1$ of the first family is shown as a solid curve,
%
%\marginpar{Valores no corresponden a la figura}
%
which is almost a straight line,
% is the solid curve 
connecting % containing 
the states $\Phiinf_1 = (0.4,\,0)^{\mathrm{T}}$ and $\Phiinf_2 = (0,\,0.5)^{\mathrm{T}}$, compare Lemma~\ref{cla:inflection}.
%
%\marginpar{upper-right hand side of $\mathcal{I}_1$ ???donde???}
%
% }
Therefore, a wave curve starting at a state on the upper-right hand side of $\mathcal{I}_1$
% , see Definition~\ref{def:inflection} of inflection curve, 
follows a centered rarefaction of the first family until the maximum packing manifold $\partial^\infty$ is reached. 
The final rarefaction point is $(0,\,1)^{\mathrm{T}}$, unless the starting state is on the $\partial^1$ axis, in such a case, the final rarefaction point is $(1,\,0)^{\mathrm{T}}$.

The characteristic velocities near the contact manifold $\mathcal{C}(\Phimax)$ are close to zero, and the characteristic directions are close to $r_2 = r_1 = (1,\,-1)^{\mathrm{T}}$, see \eqref{ev2}.
%
% Parece que es no clave para la caracterizaci—n
% The rarefactions starting at points $\Phi_\epsilon := (0.4,\,\epsilon)^{\mathrm{T}}$, for small positive $\epsilon$, pass close to the vertex $(1,\,0)^{\mathrm{T}}$. 
% Therefore, we notice that the solution profiles are similar for the state $\Phiinf_1$ and those given for the states $\Phi_\epsilon$. 
% The trajectories given for the states $\Phi_\epsilon$ contain all the points close to the $\partial^1$ axis for $\phi_1$ larger than $0.4$; in all this region the solution changes smoothly as the initial datum $\Phi^-$ does.
%
%
% {\color{red}
%
%\marginpar{crosses: de donde a donde}
%
% Please refer again to Fig.~\ref{fig:curves2x2}.
If the left state $\Phi$ is on the lower-left side of the inflection $\mathcal{I}_1$ 
% from the left-bottom corner to the right-upper side the solution behavior changes. 
then the wave is obtained by a backward 1-wave construction. % As the starting state crosses the inflection manifold, 
% The rarefaction comprises part of the previous rarefaction trailing a characteristic shock that crosses the inflection manifold. % }
%
% Thus, as the Hugoniot locus $\mathcal{H}(O)$ contains the characteristic line $\mathcal{C}(\Phis)$ the position of which is on the right of $\mathcal{I}_1$, we remark
%
%\marginpar{realmente correcto?}
%
Namely, all shocks from a state $\Phi$ at the lower-left hand side of $\mathcal{I}_1$ are connected to a state $\Phi^M \in \mathcal{H}(\Phi)$ on the right of $\mathcal{I}_1$ satisfying $\sigma(\Phi,\,\Phi^M) = \lambda_1(\Phi^M)$.
% ÀÀÀÀ no se ve de donde viene eso ????
%, but before crossing of $\mathcal{C}(\Phis)$, {\it i.e.}, $\phi_1^M + \phi_2^M < \phis$ holds. 
%
Therefore, the solution for such a state $\Phi$ comprises the 1-Lax shock from $\Phi$ to $\Phi^M$ which is characteristic at $\Phi^M$, and the rarefaction curve from $\Phi^M$ to $(0,\,1)^{\mathrm{T}}$, or $(1,\,0)^{\mathrm{T}}$ if $\Phi$ belongs to $\partial^1$.

\ \\
\noindent
{\it Construction of Example 2}. 
% {\color{red} 
The construction of Example 2 is visualized in Figure~\ref{fig:Solutions}.
The integral curves of the first family change their growth direction at two detached parts of the inflection manifold $\mathcal{I}_1$, namely $\mathcal{T}$ and $\mathcal{B}$, see Figure~\ref{fig:Solutions}~(a).
This growth direction change is the key in the distinct structure of the solution of $RP(\Phi,\,\Phimax)$ solutions in both Examples. 

\begin{figure}[htb]
\centering
\includegraphics[width = \textwidth]{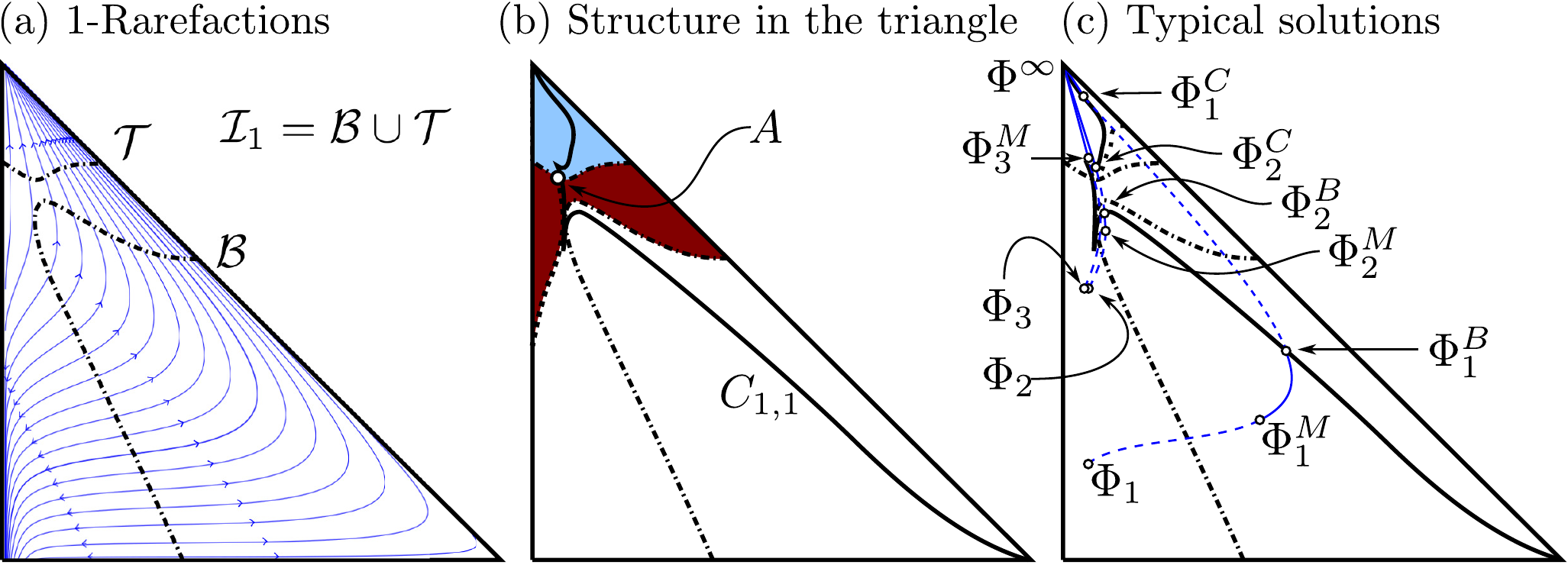}
\caption{Curves of Example 2. In (a) the first family rarefaction curves are plotted as continuous curves, the arrows point into the direction of increasing eigenvalues. The dash-dotted curves are the inflection manifold $\mathcal{I}_1$ splitted into the two branches $\mathcal{T}$ and $\mathcal{B}$. In (b) the solid curve $C_{1,1}$ is the double-contact of the first family (notice the two components, one in the light region and one another in the white region); $A$ is the state where a limit shock curve from $\mathcal{T}$ does not cross $\mathcal{B}$. In the presentation of the wave curve solutions in (c) the continuous curves correspond to rarefaction fans and the dashed curves correspond to shock waves.}
\label{fig:Solutions}
\end{figure}
%\marginpar{Falta colocar los puntos y detalles en las figuras. Las soluciones plotadas son para $\Phi = (.05, .02)^T$, $(0.04, 0.55)$, $(0.05, 0.55)$}

In Fig.~\ref{fig:Solutions}~(b) the % two 
light and dark shaded regions represent right states $\Phi$ for which the solutions are similar to that of Example 1:
%
% First of all let us explain why the solutions of the shaded regions are analogous.
%

% Any rarefaction reaching the state $\Phimax = (1,\,0)^{\mathrm{T}}$ has increasing eigenvalues $\lambda_1$.
\begin{enumerate}
\item When the states belong to the upper corner %, {\it i.e.} 
above $\mathcal{T}$ % (see Fig.~\ref{fig:Solutions}~(a))  
inside the light shaded region, then the wave curve comprises 
a single 1-rarefaction
connecting the state $\Phi$ 
moving along increasing eigenvalues $\lambda_1$
towards $\Phimax = (1,\,0)^{\mathrm{T}}$; 
% solely such a rarefaction.
%
% (1) contains the $\Phi$ states such that the solution comprises 

% (2)
% If the left state is outside of the light shaded region of Fig.~\ref{fig:Solutions}~(b) then a shock wave is needed. 
%
\item For a left state $\Phi$ in the dark shaded region, a characteristic 1-Lax shock 
to a middle state in the light shaded region crosses $\mathcal{T}$, 
from this middle state a first family rarefaction follows to the maximum package concentration. % Therefore, it only remains to construct 
\end{enumerate}

% Then, the 1-Lax shocks from any state in the dark shaded region of Fig.~\ref{fig:Solutions}~(b) may cross $\mathcal{T}$ and thus, if the state is characteristic at the middle state in the light shaded region, then it can be trailed by a rarefaction of the first family.
%
% no es obvio de que regi—n es el $\Phi$
% The backward shock curve emanating from $\mathcal{T}$ at $A$ is tangential to $\mathcal{B}$ and encloses a new feature, say 1-Lax shocks that do not cross $\mathcal{B}$.
%
%
%

% In the sequel the wave curve for left states in the white region of Fig.~\ref{fig:Solutions}~(b) are constructed.
The construction of $RP(\Phi,\,\Phimax)$ solutions for $\Phi$ in the white region of Figure~\ref{fig:Solutions}~(b) is outlined in the sequel.
%
% However, they can be classified by the first 1-Lax shock from $\Phi$; we look for right characteristic shocks which reach states enclosed by the inflection manifold $\mathcal{I}_1$. 
%
% For a state $\Phi^M$ such that $\Phi^M$ belongs to $\mathcal{H}(\Phi)$ and $\sigma(\Phi,\,\Phi^M) = \lambda_1(\Phi^M)$
%
There are three cases of right characteristic 1-Lax shocks that connect a state $\Phi$ 
to a state $\Phi^M \in \mathcal{H}(\Phi)$ such that $\sigma(\Phi,\,\Phi^M) = \lambda_1(\Phi^M)$ : 

(1) The shock curve crosses once $\mathcal{B}$;

(2) The shock curve crosses once $\mathcal{T}$;

(3) The shock curve crosses twice $\mathcal{B}$ and once $\mathcal{T}$.

\noindent

In Figure~\ref{fig:Solutions}~(c) the left states of the cases (1) are represented by $\Phi_1$ and $\Phi_2$, the case (2) is represented by point 2 above and, case (3) is represented by $\Phi_3$.
The wave curves starting from $\Phi_1$, $\Phi_2$ and $\Phi_3$ have the following structure:
\begin{eqnarray}
\label{struct1}
\Phi_1 \xrightarrow{\mathrm{1-Lax}} \big|\Phi_1^M \xrightarrow{\mathrm{1-rar}} \Phi_1^B\big| \xrightarrow{\mathrm{1-Lax}} \big|\Phi_1^C \xrightarrow{\mathrm{1-rar}} \Phimax, \\
\label{struct2}
\Phi_2 \xrightarrow{\mathrm{1-Lax}} \big|\Phi_2^M \xrightarrow{\mathrm{1-rar}} \Phi_2^B\big| \xrightarrow{\mathrm{1-Lax}} \big|\Phi_2^C \xrightarrow{\mathrm{1-rar}} \Phimax, \\
\label{struct3}
\Phi_3 \xrightarrow{\mathrm{1-Lax}} \big|\Phi_3^M \xrightarrow{\mathrm{1-rar}} \Phimax,
\end{eqnarray}
where the symbol $\big|$ indicates where a shock is characteristic.
For case (3) the construction of the shock curves proceeds as before, 
namely a 1-Lax shock followed by a 1-rarefaction connecting a middle state $\Phi^M$ in the light shade region in Fig.~\ref{fig:Solutions}~(b) with $\Phi^\infty$. 
%
% Let $\Phi$ be a state in the situation (1).
For case (i), with $i = 1$ or $2$, the characteristic shock $(\Phi_i,\,\Phi_i^M)$ precedes a 1-rarefaction from $\Phi_i^M$ towards $\mathcal{B}$ at a state $\Phi_i^B$ on $C_{1,1}$,
% \Phi_1^{ \mathcal{B}}
from there another 1-Lax left and right characteristic shock connects to a state $\Phi_i^C$ on the other side of $C_{1,1}$.
%, where the shock $(\Phi^B,\,\Phi^C)$ is 1-Lax characteristic at both sides.
From $\Phi_i^C$ the wave curve is terminated by the 1-rarefaction to $\Phimax$. 

% Once again, consider $\Phi_2$ in Fig.~\ref{fig:Solutions}(c) as a moving state to $\Phi_3$ continuously.

If the state $\Phi_2$ is approximated to a state $\Phi_3$ then a continuous change in the solution profile is observed
such that the wave group \eqref{struct2} collapses to the wave group \eqref{struct3}.
Indeed, if $\Phi_2$ comes closer to $\Phi_3$, then $\Phi_2^M$ comes closer to $\Phi_2^B$ and $\Phi_2^C$ to $\Phi_3^M$ in such a way that the shock speeds $\sigma(\Phi_2,\,\Phi_2^M)$ and $\sigma(\Phi_2^B,\,\Phi_2^C)$ approximate $\sigma(\Phi_3,\,\Phi_3^M)$, until, in the limit, the 1-rarefaction wave from $\Phi_2^M$ to $\Phi_2^B$ disappears.

%%%%%%%%%%%%%%%%%%%%%%%%%%%%%%%%%%%%%%%%%%%%%%%%%%%%%%%%%%%%%%%%%%%%%%%%%%%%%
%%%%%%%%%%%%%%%%%%%%%%%%%% F I N : %%%%%%%%%%%%%%%%%%%%%%%%%%%%%%%%%%%%%%%%%%
%%%%%%%%%%%%%%%%%%%%%%%%%%%%%%%%%%%%%%%%%%%%%%%%%%%%%%%%%%%%%%%%%%%%%%%%%%%%%
%%%%%%%%%%%%%%%%%%% Soluciones nuevas, casos numericos %%%%%%%%%%%%%%%%%%%%%%
%%%%%%%%%%%%%%%%%%%%%%%%%%%%%%%%%%%%%%%%%%%%%%%%%%%%%%%%%%%%%%%%%%%%%%%%%%%%%
%%%%%%%%%%%%%%%%%%%%%%%%%%%%%%%%%%%%%%%%%%%%%%%%%%%%%%%%%%%%%%%%%%%%%%%%%%%%%
%%%%%%%%%%%%%%%%%%%%%%%%%%%%%%%%%%%%%%%%%%%%%%%%%%%%%%%%%%%%%%%%%%%%%%%%%%%%%

% \appendix  \input{appendix}

%***************************************************************************
\section*{Acknowledgments}
The authors are grateful to~Dan Marchesin for allow us to
use the {\tt RPn} package that the Fluid Dynamics Laboratory at
IMPA is developing.
The first author is supported by Conicyt (Chile) through Fondecyt
project \#~1120587. 
The second author was partially supported by FAPERJ (Brazil) through the
fellowship grant E-26/102.474/2010. % <- new, old: 11080253.

%%***************************************************************************

% \pagebreak

\end{document}